\documentclass{amsart}

\usepackage{amssymb}

\usepackage{graphicx}

\newtheorem{theorem}{Theorem}[section]
\newtheorem{lemma}[theorem]{Lemma}

\theoremstyle{definition}

\newtheorem*{remark}{Remark}

\numberwithin{equation}{section}

\begin{document}
\title[Loewner chains on Riemann surfaces]{Loewner chains on the universal covering space of a Riemann surface}

\author[J. Tsai]{Jonathan Tsai}
\address{Department of Mathematics \\ Chinese University of Hong Kong \\ Shatin \\ New Territories \\ Hong Kong}
\email{jtsai@math.cuhk.edu.hk}
\thanks{The author would like to express his gratitude to K. S. Lau for funding his postdoctoral research. In addition, the author would like to thank his PhD examiners Alan Beardon, Phil Rippon and his PhD supervisor Keith Carne for their helpful suggestions without which this paper could not have been written.}


\begin{abstract}
Let $R$ be a hyperbolic Riemann surface with boundary $\partial R$ and suppose that $\gamma:[0,T]\rightarrow R\cup\partial R$ is a simple curve with $\gamma(0,T]\subset R$ and $\gamma(0)\in\partial R$. By lifting $R_{t}=R\setminus \gamma(0,t]$ to the universal covering space of $R$ (which we assume is the upper half-plane $\mathbb{H}=\{z\in\mathbb{C}:\mathrm{Im}[z]>0\}$) via the covering map $\pi:\mathbb{H}\rightarrow R$, we can define a family of simply-connected domains $D_{t}=\pi^{-1}(R_{t})\subset\mathbb{H}$. For each $t\in[0,T]$, suppose that $f_{t}$ is a conformal map of $\mathbb{H}$ onto $D_{t}$ such that $f(z,t)=f_{t}(z)$ is differentiable almost everywhere in $(0,T)$ with respect to $t$. In this paper, we will derive a differential equation that describes how $f(z,t)$ evolves in time $t$. This should be viewed as an extension of the Loewner differential equation to curves on Riemann surfaces with boundary.

The motivation of this paper is the desire to extend Schramm's stochastic Loewner evolution to multiply-connected domains and Riemann surfaces.
\end{abstract}

\maketitle
\section{Introduction}
In 1923, Charles Loewner (then known as Karl L\"{o}wner) invented a method of studying conformal mappings onto domains slit by a growing curve (see  \cite{MR1512136}). More recently, interest in the Loewner differential equation has been rekindled by Oded Schramm's invention of stochastic Loewner evolution (SLE) in \cite{MR1776084}. A key ingredient in Schramm's work is a version of the Loewner differential equation called the chordal Loewner differential equation: let $\mathbb{H}=\{z\in\mathbb{C}:\mathrm{Im}[z]>0\}$ denote the upper half-plane and suppose that $\gamma:[0,T]\rightarrow \overline{\mathbb{H}}$ is a simple curve with $\gamma(0)\in\mathbb{R}$ and $\gamma(0,T]\subset\mathbb{H}$. Then the Riemann mapping theorem implies that there exists a family of conformal maps $\{f_{t}\}_{t\in[0,T]}$ such that each $f_{t}$ is a conformal map of $\mathbb{H}$ onto $\mathbb{H}\setminus \gamma(0,t]$ and $f_{t}$ satisfies the hydrodynamic normalization: $f_{t}(\infty)=\infty$ and
\[f_{t}(z)=z-\frac{c(t)}{z} + O\left(\frac{1}{z^{2}}\right) \text{ as } z\rightarrow \infty.\]
Then $c(t)$ is strictly increasing and positive thus we can reparameterize such that $c(t)=2t$ (the quantity $c(t)$ is known as the half-plane capacity of $\gamma(0,t]$). With this normalization and parametrization, the function $f(z,t)=f_{t}(z)$ is differentiable with respect to $t$ and moreover
\[\dot{f}_{t}(z)=\frac{-2f_{t}'(z)}{z-\xi(t)},\]
where $\xi(t)=f_{t}^{-1}(\gamma(t))$. Here and throughout this paper we will use the notation
\[\dot{f}_{t}(z)=\frac{\partial}{\partial t} f_{t}(z) \text{ and } f_{t}'(z)=\frac{\partial}{\partial z} f_{t}(z).\]This differential equation is called the \emph{chordal Loewner differential equation} and $\xi(t)$ is called the \emph{chordal driving function} of the curve $\gamma$. Chordal stochastic Loewner evolution with parameter $\kappa$ is defined to be the random family of conformal maps $f_{t}$ obtained as the solution of the Loewner differential equation  with driving function $\sqrt{\kappa}B_{t}$ for some $\kappa\geq 0$, where $B_{t}$ is standard 1-dimensional Brownian motion. It turns out that, in some sense, SLE are the only conformally invariant 2-dimensional processes. This is important because many lattice models in probability and statistical physics are conjectured to have conformally invariant scaling limit. Since its invention, many previously intractable problems in the field have been solved rigorously using SLE. For full details on the construction of the chordal Loewner differential equation and SLE, see \cite{MR2129588} and \cite{MR2079672}.

One of the problems in the area is how to define SLE on a multiply-connected domain or Riemann surface (see \cite{MR2334202}). There are several papers which consider this problem: in \cite{MR2128237} and \cite{MR2268538}, Dapeng Zhan looks at SLE in an annulus defined via the Schwarz kernel formula. Also, in his PhD thesis, \cite{Zhanthesis}, Dapeng Zhan invented a measure on curves on Riemann surfaces which in the simply-connected case, corresponds to SLE. Finally, in \cite{MR2230350}, Robert Bauer and Roland Friedrich define SLE on planar multiply-connected domains via Komatu's generalization of the Loewner differential equation to multiply-connected domains given in \cite{MR0046437}. In this paper we will introduce a Loewner differential equation on the covering space of a Riemann surface with boundary. This can be used to define SLE on Riemann surfaces.

\ \\
More specifically, we will do the following: suppose that $R$ is a multiply-connected domain or Riemann surface with boundary $\partial R$, then the uniformization theorem implies that $R$ is conformally equivalent to $\mathbb{H}/\Gamma$ where $\Gamma$ is a Fuchsian group that preserves $\mathbb{H}$ and contains no elliptic M\"{o}bius transformations (see \cite[Chapter 10]{MR0357743}). Suppose that $\gamma:[0,T]\rightarrow R\cup\partial R$ is a simple curve such that $\gamma(0)\in \partial R$ and $\gamma(0,T]\subset R$. Let $\pi:\mathbb{H}\rightarrow R$ be the covering map such that group of deck transformations (automorphisms) of $\pi$ is $\Gamma$. We can lift $\gamma$ via $\pi$ to a curve $\widetilde{\gamma}:[0,T]\rightarrow\mathbb{H}$ with $\widetilde{\gamma}(0,T]\subset\mathbb{H}$ and $\widetilde{\gamma}(0)\in\mathbb{R}$. Then
\[K_{t}=\pi^{-1}(\gamma(0,t])=\bigcup_{\phi\in\Gamma}\phi(\widetilde{\gamma}(0,t]).\]
Moreover, for $t\in[0,T]$, $\mathbb{H}\setminus K_{t}$ is a simply-connected domain so by the Riemann mapping theorem, there exists a family of conformal bijections $f_{t}$ that map $\mathbb{H}$ conformally onto $\mathbb{H}\setminus K_{t}$. For each $t \in[0,T]$, the function $\pi\circ f_{t}:\mathbb{H}\rightarrow R\setminus \gamma(0,t]$ is a universal covering map of $R\setminus \gamma(0,t]$. Also, the family of functions $f_{t}$ induces a family of Fuchsian groups $\Gamma_{t}=f_{t}^{-1}\circ\Gamma\circ f_{t}=\{f_{t}^{-1}\circ\phi\circ f_{t}:\phi\in\Gamma\}$ such that $R\setminus\gamma(0,t]$ is conformally equivalent to $\mathbb{H}/\Gamma_{t}$.  We will see later that if we normalize the functions $f_{t}$ appropriately, then $f(z,t)=f_{t}(z)$ is differentiable with respect to $t$ at least for almost every $t\in(0,T)$. We will call the family of functions $\{f_{t}\}_{t\in[0,T]}$ a \emph{Loewner chain corresponding to the curve $\gamma$ on $R$}. We will prove that in the case where $R$ is a finite Riemann surface (so that $\Gamma$ is finitely-generated), $f_{t}$ satisfies a version of the Loewner differential equation. Without loss of generality, we will always assume that $R=\mathbb{H}/\Gamma$ and $\pi:\mathbb{H}\rightarrow\mathbb{H}/\Gamma$ is the quotient map; the result for general Riemann surfaces follows from conformal equivalence.

\begin{theorem}\label{genslitRS}
Suppose that $\Gamma$ is a finitely-generated Fuchsian group preserving $\mathbb{H}$ such that $R=\mathbb{H}/\Gamma$ is a Riemann surface with boundary $\partial R$.

Let $\gamma:[0,T]\rightarrow R\cup \partial R$ be a simple curve such that $\gamma(0,T]\subset R$ and $\gamma(0)\in\partial R$. Let $\{f_{t}\}$ be a Loewner chain corresponding to the curve $\gamma$ on $R$ and let $\Gamma_{t}=f_{t}^{-1}\circ\Gamma\circ f_{t}$. Also let $\phi_{1,t},\ldots,\phi_{n,t}$ be the hyperbolic free generators of $\Gamma_{t}$ such that each $\phi_{k,t}$ is differentiable with respect to $t$ for almost all $t\in(0,T)$. Suppose that we can find $c\in\mathbb{R}$ such that $c$ is not in the limit set of $\Gamma_{t}$ for all $t\in[0,T]$. Then, $f_{t}$ satisfies
\[\dot{f}_{t}(z)=f_{t}'(z)P(z,t),\]
where
\normalsize\[P(z,t)=\frac{\lambda(t)+\sum_{\phi_{t}\in\Gamma_{t}}\left[\frac{1}{\phi_{t}(z)-\xi(t)} + \sum_{k=1}^{n} \delta_{k}(t)\left(\frac{\dot{\phi}_{t}(z)-\dot{\phi}_{k,t}(c)}{\phi_{t}(z)-\phi_{k,t}(c)}-\frac{\dot{\phi}_{t}(z)}{\phi_{t}(z)-c} \right)  \right]}{\sum_{\phi_{t}\in\Gamma_{t}}\sum_{k=1}^{n}\delta_{k}(t)\left(\frac{\phi_{t}'(z)}{\phi_{t}(z)-\phi_{k,t}(c)}-\frac{\phi_{t}'(z)}{\phi_{t}(z)-c}\right)}\]\normalsize
Here, $\xi(t)$ satisfies $\pi\circ f_{t}^{-1}(\xi(t))=\gamma(t)$; $\lambda(t)$ and $\delta_{1}(t),\ldots,\delta_{n}(t)$ are real-valued functions.
\end{theorem}
We will see that the functions $\delta_{1}(t),\ldots,\delta_{n}(t)$ in Theorem \ref{genslitRS} can be found by solving a system of linear equations (Lemma \ref{lem8}) and they ensure that the sum in the numerator of $P(z,t)$ converges.

The idea of the proof of Theorem \ref{genslitRS} is as follows: we first study the properties of the function
\[P(z,t)=\frac{\dot{f}_{t}(z)}{f_{t}'(z)}.\]
 We will then construct a function that satisfies these properties. The theory of automorphic functions (see \cite{Ford} or \cite{MR0164033}) will then guarantee that our function only differs from $P(z,t)$ by the function $\lambda(t)$.

We will also show that, as in the case of the radial Loewner differential equation and the chordal Loewner differential equation, the converse is also true as well i.e. given functions $\xi(t)$ and $\lambda(t)$, we can solve the differential equation
\[\dot{g}_{t}(z)=-P(g_{t}(z),t) \text{ with initial condition } g_{0}(z)\equiv z,\]
where $P(z,t)$ is given in Theorem \ref{genslitRS}, to get a family of conformal maps $\{g_{t}\}$. Each $g_{t}$ maps a domain $H_{t}\subset\mathbb{H}$ conformally onto $\mathbb{H}$. In addition, $H_{t}$ is $\Gamma$-invariant (i.e. $\phi(H_{t})=H_{t}$ for all $\phi\in\Gamma$) and hence if $\pi:\mathbb{H}\rightarrow R$ denotes the universal covering map whose automorphism group if $\Gamma$, then $\{\pi(H_{t})\}$ is a family of sub-Riemann surfaces of $R$ and $\pi\circ g_{t}^{-1}:\mathbb{H}\rightarrow \pi(H_{t})$ is a universal covering map. The additional difficulty in our situation compared to the simply-connected case is that we need to be able to determine the family of Fuchsian groups $\{\Gamma_{t}\}$ beforehand in order to be able to solve the differential equation to find the conformal mappings $g_{t}$. This corresponds to the fact that unlike in the simply-connected case, the Teichm\"{u}ller space of $R$ is non-trivial. This is related to \cite[Section 6]{MR2230350}.

\ \\
The rest of this paper is structured as follows:
\begin{enumerate}
\item[\S 2:] We will recall without proof some facts from the theory of Fuchsian groups and automorphic functions.
\item[\S 3:] We  will look at some properties that the function $P(z,t)$ satisfies.
\item[\S 4:] The results from the previous section will be used to prove Theorem \ref{genslitRS}
\item[\S 5:] We will consider the special case where $\Gamma$ is generated by a single hyperbolic M\"{o}bius transformation and hence $\mathbb{H}\setminus\Gamma$ is conformally equivalent to an annulus. In this case, the differential equation has a simpler form.
\item[\S 6:] We will prove that the converse of Theorem \ref{genslitRS}.
\item[\S 7:] Finally, we will make some brief comments on how stochastic Loewner evolution on Riemann surfaces can be defined using our results.
\end{enumerate}
\section{Fuchsian groups and automorphic functions}
In this section, we will recall some basic definitions and results which we will need in this paper. References to the literature will be provided where appropriate.

Recall that a M\"{o}bius transformations $\phi:\widehat{\mathbb{C}}\rightarrow\widehat{\mathbb{C}}$ (where $\widehat{\mathbb{C}}$ denotes the Riemann sphere $\mathbb{C}\cup\{\infty\}$) is a function of the form
\[\phi(z)=\frac{az+b}{cz+d} \text{ for } a,b,c,d\in\mathbb{C} \text{ with } ad-bc\neq 0.\]

We say that a M\"{o}bius transformation $\phi$ preserves the upper half-plane $\mathbb{H}$ if $\phi(\mathbb{H})=\mathbb{H}$. An equivalent condition for $\phi$ to preserve $\mathbb{H}$ is $a,b,c,d\in\mathbb{R}$. Note that the set of all M\"{o}bius transformations that preserve $\mathbb{H}$ form a group under composition which we denote $\mathcal{M}$.

We can classify the M\"{o}bius transformations that preserve $\mathbb{H}$ by their fixed points (i.e. points $c\in\widehat{\mathbb{C}}$ such that $\phi(c)=c$). If $\phi$ is a M\"{o}bius transformation that preserves $\mathbb{H}$, then
\begin{enumerate}
\item $\phi$ is \emph{elliptic} if it has exactly 2 fixed points in $\widehat{\mathbb{C}}$ that are not contained in $\partial\mathbb{H}=\mathbb{R}\cup\{\infty\}$ ;
\item $\phi$ is \emph{parabolic} if it has a unique fixed point in $\partial\mathbb{H}$;
\item $\phi$ is \emph{hyperbolic} if it has exactly 2 fixed points in $\partial\mathbb{H}$.
\end{enumerate}
These are the only possibilities for $\phi$ (see \cite[p. 67]{MR1393195}).

\ \\
A \emph{Fuchsian group} $\Gamma$ is a subgroup of $\mathcal{M}$ such that the induced topology on $\Gamma$ from the topology of local uniform convergence on $\mathcal{M}$ is the discrete topology i.e. the group is \emph{discrete} (see \cite[Section 2.3]{MR1393195}). The limit set of a Fuchsian group $\Gamma$ is the closure of the set of all fixed points of hyperbolic M\"{o}bius transformations contained in $\Gamma$. An equivalent definition for the limit set $\Lambda$ is the following: for $z\in\mathbb{H}$, $\Lambda$ is the set of all $w$ such that $\phi_{n}(z)\rightarrow w$ where $(\phi_{n})$ is a  sequence of distinct M\"{o}bius transformations in $\Gamma$.
This does not depend on the choice of $z\in\mathbb{H}$. See \cite[p.96-98]{MR1393195} for more details. From the first definition, it is clear that $\Lambda\subseteq \partial\mathbb{H}$. If $\Lambda\neq \partial\mathbb{H}$, then we say that $\Gamma$ is a \emph{Fuchsian group of the second kind}. If $\Gamma$ is a Fuchsian group of the second kind, then it can also be shown that if $x\in\mathbb{R}\setminus\Lambda$, then the set of limit points of $x$ is $\Lambda$ (see \cite[Theorem 5.3.9]{MR1393195}). We will need the following result.
\begin{lemma}\label{lempoin}
If $\Gamma$ is a Fuchsian group of the second kind such that $\infty\not\in\Lambda$, then the series
\[\sum_{\phi\in\Gamma} |\phi'(z)|\]
converges locally uniformly in $\widehat{\mathbb{C}}\setminus\Lambda$.
\end{lemma}
\begin{proof}
See \cite[Section 47]{Ford}.
\end{proof}
Note that if $\Gamma$ is a Fuchsian group of the second kind, then $\widehat{\mathbb{C}}/\Gamma$ is a Riemann surface (see \cite[Theorem 6.2.1]{MR1393195}).

\ \\
Let $\Gamma$ be a Fuchsian group of the second kind with limit set $\Lambda$. Suppose that $\Psi$ is a meromorphic function on $\mathbb{H}$. We say that $\Psi$ is \emph{character-automorphic with respect to $\Gamma$} if
\[\Psi(\rho(z))=\chi(\rho)\Psi(z) \text{ for all }\rho\in\Gamma\]
for some $\chi(\rho)\in\mathbb{C}$. $\chi$ is called the \emph{character} of $\Psi$ and it is easy to see that it satisfies
\[\chi(\rho_{1}\circ\rho_{2})=\chi(\rho_{1})\chi(\rho_{2})\]
for any $\rho_{1},\rho_{2}\in\Gamma$.

A character-automorphic function is an \emph{automorphic} function if $\chi(\rho)\equiv 1$. We will be interested in a special class of automorphic functions: $\Psi$ is a \emph{simple automorphic function with respect to $\Gamma$} if it satisfies
\begin{enumerate}
\item[(i)] $\Psi$ is meromorphic function on $\widehat{\mathbb{C}}\setminus\Gamma$;
 \item[(ii)] $\Psi$ is automorphic with respect to $\Gamma$;
 \item[(iii)] if $p$ is the fixed point of a parabolic M\"{o}bius transformation in $\Gamma$, then the limit of $\Psi(z)$ as $z\rightarrow p$ inside a fundamental region of $\Gamma$ (see \cite[p. 37]{Ford}) exists (the limit may be infinite).
\end{enumerate}
Note that any simple automorphic function defines a meromorphic function on the Riemann surface $\widehat{\mathbb{C}}/\Gamma$. We will need the following important results on simple automorphic functions.
\begin{lemma}\label{lemauto}
Suppose that $\Gamma$ is a Fuchsian group of the second kind with fundamental region $F$ and $\Psi$ is a simple automorphic function with respect to $\Gamma$. Then
\begin{enumerate}
\item[(i)] the number of zeroes of $\Psi$ contained in $F$ is the same as the number of poles of $\Psi$ contained in $F$ (counting multiplicities);
\item[(ii)] if $\Psi$ is not constant, then $\Psi$ takes every complex value the same number of times in each fundamental region;
\item[(iii)] if $\Psi$ has no poles in $F$, then $\Psi$ is constant;
\item[(iv)] if $\Psi$ has exactly one pole in $F$, then $\Psi$ is constant.
\end{enumerate}
\end{lemma}
\begin{proof}
For the proofs of (i), (ii) and (iii), see \cite[Section 42]{Ford}. For (iv), note that if $\Psi$ is a non-constant simple automorphic function with exactly 1 pole in $F$, then by (i) and (ii), it defines a conformal map of $\widehat{\mathbb{C}}/\Gamma$ onto $\widehat{\mathbb{C}}$. This is a contradiction because $\widehat{\mathbb{C}}/\Gamma$ cannot be conformally equivalent to $\widehat{\mathbb{C}}$.
\end{proof}

\section{Properties of the function $P(z,t)$}
Let $D\subsetneq\mathbb{C}$ be a simply-connected domain. We define a \emph{hull} in $D$ to be a closed set $K\subset \overline {D}$ such that $D\setminus K$ is also a simply-connected domain. Suppose that $\{K_{t}\}_{t\in[0,T]}$ is a family of hulls in a simply-connected domain $D\subsetneq\mathbb{C}$ and let  $D_{t}=D\setminus K_{t}$. We will call the family $\{K_{t}\}$ a \emph{continuously growing family of hulls} in $D$ if it satisfies the following properties:
\begin{enumerate}
\item $K_{0}=\emptyset$; $K_{T}\subsetneq D$;
\item $K_{s}\subsetneq K_{t}$ for $s< t$;
\item For $s\in[0,T]$, $D_{t}\rightarrow D_{s}$ as $t\rightarrow s$; this convergence is in the sense of Carath\'{e}odory kernel convergence with respect to any point in $D\setminus K_{T}$ (see \cite[p. 29]{MR0507768}).
\end{enumerate}
\begin{lemma}\label{lem0}
Suppose that $\{K_{t}\}$ is a continuously growing family of hulls in $D$. Then there exists a family of conformal maps $\{f_{t}\}$ such that each $f_{t}$ is a conformal map of $D$ onto $D\setminus K_{t}$ and moreover $f(z,t)=f_{t}(z)$ is differentiable with respect to $t$ almost everywhere in $(0,T)$.
\end{lemma}
 We will call the family of functions $\{f_{t}\}$ in Lemma \ref{lem0} a \emph{generalized Loewner chain} in $D$ corresponding to the continuously growing family of hulls $\{K_{t}\}$.
\begin{proof}[Proof of Lemma \ref{lem0}]
We first assume that $D=\mathbb{D}= \{z\in\mathbb{C}:|z|<1\}$ the unit disc and $0\not\in K_{t}$  for all $t\in[0,T]$. Let $f_{t}$ be the unique conformal map of $\mathbb{D}$ onto $D_{t}$ such that $f_{t}(0)=0$ and $f_{t}'(0)=r(t)>0$ (which exists by the Riemann mapping theorem).  Then note that properties (2) and (3) above imply that $r(t)$ is continuous and strictly increasing. Then the proof of Lemma 6.1 in \cite{MR0507768} shows that
\[|f_{t}(z)-f_{s}(z)|\leq \frac{C|z|}{(1-|z|)^{4}}(r(t)-r(s))\]
for all $z\in\mathbb{D}$ and for some constant $C>0$. In particular, this implies that $f_{t}(z)$ is absolutely continuous with respect to $t$ and so by Fatou's Theorem is almost everywhere differentiable with respect to $t$.

In the general case, by property (1), we can find $w\in D\setminus K_{T}$. By the Riemann mapping theorem, there exists a conformal map $\psi:D\rightarrow \mathbb{D}$ with $\psi(w)=0$. We let $\widetilde{K}_{t}=\psi(K_{t})$. Then $\widetilde{K}_{t}$ is a continuously growing family of hulls in $\mathbb{D}$ such that $0\not\in\widetilde{K}_{t}$ for all $t\in[0,T]$. Hence by the previous part, we can find conformal maps $\widetilde{f}_{t}$ of $\mathbb{D}$ onto $\mathbb{D}\setminus \widetilde{K}_{t}$ such that $\widetilde{f}(z,t)=\widetilde{f}_{t}(z)$ is differentiable for almost all $t$. We define
$f_{t}=\psi^{-1}\circ \widetilde{f}_{t}\circ \psi$. Then $f_{t}$ is a conformal map of $D$ onto $D\setminus K_{t}$ and moreover $f(z,t)=f_{t}(z)$ is differentiable almost everywhere with respect to $t$.
\end{proof}
\begin{remark}
Note that if $\phi_{t}$ is a family of conformal maps such that for fixed $t$, $\phi_{t}$ is a conformal automorphism of $D$ and also $\phi_{t}(z)$ is differentiable with respect to $t$ for almost every $t\in(0,T)$. Then $F_{t}=f_{t}\circ\phi_{t}$ is also a family of conformal mappings of $D$ onto $D_{t}$ such that $f_{t}(z)$ is differentiable for almost every $t$. Thus we conclude that there are many different normalizations of $f_{t}$ such that $f_{t}(z)$ is differentiable for almost every $t\in(0,T)$.
\end{remark}
We return to the situation we are interested in. Let $\Gamma$ be a finitely generated Fuchsian group such that $R=\mathbb{H}/\Gamma$ is a Riemann surface with boundary $\partial R$ and let $\pi:\mathbb{H}\rightarrow R$ be the quotient map. We define $K$ to be a \emph{hull} in $R$ if $K$ is a closed set in $R$ such that
$\pi^{-1}(K)$ is a hull in $\mathbb{H}$. If $\{K_{t}\}_{t\in[0,T]}$ is a family of hulls on $R$, we say that $\{K_{t}\}_{t\in[0,T]}$ is a \emph{continuously growing family of hulls} if $\{\pi^{-1}(K_{t})\}_{t\in[0,T]}$ is a continuously growing family of hulls in $\mathbb{H}$.

So suppose that $\{K_{t}\}_{t\in[0,T]}$ is a continuously growing family of  hulls in $R$. By Lemma \ref{lem0}, we can find a family of conformal maps $\{f_{t}\}_{t\in[0,T]}$ such that $f_{t}$ maps $\mathbb{H}$ onto $H_{t}=\mathbb{H}\setminus \pi^{-1}(K_{t})$ and  $f(z,t)=f_{t}(z)$ is differentiable for almost all $t\in(0,T)$. The family of functions $\{f_{t}\}_{t\in[0,T]}$ is called a \emph{Loewner chain} corresponding to the continuously growing family of hulls $\{K_{t}\}_{t\in[0,T]}$ on $R$. Then, for some set $E\subset(0,T)$ of Lebesgue measure $0$, we define the function $P:\mathbb{H}\times([0,T]\setminus E)\rightarrow \mathbb{C}$ to be
\[P(z,t)=\frac{\dot{f}_{t}(z)}{f_{t}'(z)}.\]
Note that each $H_{t}$ is $\Gamma$-invariant i.e.
\[\phi(H_{t})=H_{t} \text{ for any } \phi\in\Gamma.\]
Hence $f_{t}^{-1}\circ\phi\circ f_{t}$ is a conformal automorphism of $\mathbb{H}$ and hence is also a M\"{o}bius transformation preserving $\mathbb{H}$. Thus, if we define
\[\Gamma_{t}=f_{t}^{-1}\circ\Gamma\circ f_{t}=\{f_{t}^{-1}\circ\phi\circ f_{t}:\phi\in\Gamma\},\]
then $\Gamma_{t}$ is also a Fuchsian group. Also, note that $f_{t}^{-1}(\pi^{-1}(\partial K_{t}\cap R))\subset \partial\mathbb{H}$ is not contained in the limit set $\Lambda_{t}$ of $\Gamma_{t}$. Hence, for $t\in(0,T]$, we have $\Lambda_{t}\neq\partial \mathbb{H}$ i.e. $\Gamma_{t}$ is a Fuchsian group of the second kind.
\begin{lemma}\label{lem1}
The function $P(z,t)$ satisfies
\[P(\phi_{t}(z),t)=\phi_{t}'(z)P(z,t)-\dot{\phi}_{t}(z)\]
for any $\phi_{t}\in\Gamma_{t}$.
\end{lemma}
\begin{proof}
Note that for any $\phi_{t}\in\Gamma_{t}$, we can write $\phi_{t}=f_{t}^{-1}\circ\phi\circ f_{t}$ for some $\phi\in\Gamma$. Also note that $f_{t}$ satisfies
\[\dot{f}_{t}(z)=f_{t}'(z)P(z,t),\]
and if $g_{t}=f_{t}^{-1}$, then
\[\dot{g}_{t}(z)=-P(g_{t}(z),t).\]
Hence, the result follows from differentiating $\phi_{t}=g_{t}\circ\phi\circ f_{t}$ with respect to $t$ using the chain rule.
\end{proof}
Now let $\gamma:[0,T]\rightarrow R\cup \partial R$ be a simple curve such that $\gamma(0,T]\subset R$ and $\gamma(0)\in\partial R$. In the above notation, we let $K_{t}=\gamma(0,t]$. Then $\{K_{t}\}$ is a continuously growing family of hulls in $R$ and we define the Loewner chain $\{f_{t}\}$ and $P(z,t)$ as above. Note that we can write
\[\pi^{-1}(\gamma(0,t])=\bigcup_{\phi\in\Gamma}\phi( \widetilde{\gamma}(0,t]),\]
where $\widetilde{\gamma}:[0,T]\rightarrow \overline{\mathbb{H}}$ is a simple curve with $\widetilde{\gamma}(0)\in\mathbb{R}$ and $\widetilde{\gamma}(0,T]\subset\mathbb{H}$. Let $\xi(t)=f_{t}^{-1}(\widetilde{\gamma}(t))$.
\begin{lemma}\label{lem2}
Let $K_{t}=\gamma(0,t]$ be as above. Then $P(z,t)$ is a meromorphic function in $\widehat{\mathbb{C}}\setminus\Lambda_{t}$ and moreover, the only poles of $P(z,t)$ are simple poles of negative residue at $\phi_{t}(\widetilde{\xi}(t))$ for each $\phi_{t}\in\Gamma_{t}$. Also, the function $\xi(t)$ is continuous and the limit of $P(z,t)$ as $z$ tends to the fixed point of a parabolic element of $\Gamma$ inside a fundamental region exists.
\end{lemma}
\begin{proof}
Suppose that $t\in(0,T)$ such that $P(z,t)$ is defined and take $s<t$ sufficiently close to $t$. Let $F_{t}=f_{s}^{-1}\circ f_{t}$ and $I_{t}^{s}$ be the preimage of $f_{s}^{-1}(\pi^{-1}(\gamma(s,t)))$ under $F_{t}$. Then we can write
\[
I_{t}^{s}=\bigcup_{\phi_{t}\in\Gamma_{t}}\phi_{t}(J_{t}^{s}),\]
where $J_{t}^{s}$ is an interval on the real line containing $\xi(t)$. By the Schwarz reflection principle, $F_{t}$ can be
extended to an analytic map on $\overline{\mathbb{H}}\setminus I_{t}^{s}$. Note that by the chain rule,
\[\dot{F}_{t}(z)=F_{t}'(z)P(z,t).\]
Thus $P(z,t)$ is also analytic in $\overline{\mathbb{H}}\setminus I_{t}^{s}$. Then as $s\rightarrow t$, we have $\bigcap_{s<t}J_{t}^{s}=\{\xi(t)\}$. Hence $P(z,t)$ is holomorphic in $\overline{\mathbb{H}}\setminus\{\phi_{t}(\xi(t)):\phi_{t}\in\Gamma_{t}\}$.

For each $t\in[0,T]$, by the Riemann mapping theorem, there exists a unique conformal map $\widetilde{g}_{t}:\mathbb{H}\setminus \widetilde{\gamma}(0,t] \rightarrow \mathbb{H}$ such that $\widetilde{g}_{t}$ is normalized hydrodynamically i.e. $\widetilde{g}_{t}(\infty)=\infty$ and
\[ \widetilde{g}_{t}(z)=z+\frac{c(t)}{z}+O(z^{-2}) \text{ as } z\rightarrow\infty.\]
Here, $c(t)$ is the half-plane capacity of $\widetilde{\gamma}(0,t]$ and is strictly increasing and hence is differentiable for almost every $t\in(0,T)$. Hence for almost all $t\in(0,T)$, $\widetilde{g}_{t}$ satisfies the chordal Loewner differential equation
\[\dot{\widetilde{g}}_{t}(z)=\frac{\dot{c}(t)}{\widetilde{g}_{t}(z)-\widetilde{\xi}(t)},\]
where $\widetilde{\xi}(t)=\widetilde{g}_{t}(\widetilde{\gamma}(t))$ (see \cite[Section 4.1]{MR2129588}). Note that $\widetilde{\xi}(t)$ is continuous (see \cite[Lemma 4.2]{MR2129588}).

Let $\beta_{t}=\widetilde{g}_{t}\circ f_{t}$. Using the Schwarz reflection principle, we can extend $\beta_{t}$ to a conformal mapping of a neighbourhood of $\xi(t)$ to a neighbourhood of $\widetilde{\xi}(t)$ with
\[\beta_{t}(\xi(t))=\widetilde{\xi}(t).\]
So in particular, $\beta_{t}$ is conformal at $\xi(t)$ and thus has Taylor expansion
\[\beta_{t}(z)=\widetilde{\xi}(t) + a_{1,t}(z-\xi(t)) + a_{2,t}(z-\xi(t))^2+\cdots\]
near $z=\xi(t)$. This also implies that $\xi(t)$ is continuous.

 By the chain rule, for almost all $t\in[0,T]$,
\begin{eqnarray*}\dot{\beta}_{t}(z)&=&\dot{\widetilde{g}}_{t}(f_{t}(z)) + \widetilde{g}_{t}'(f_{t}(z))\dot{f}_{t}(z)
\\&=& \frac{\dot{c}(t)}{\beta_{t}(z)-\widetilde{\xi}(t)}+ \beta_{t}'(z)P(z,t).
\end{eqnarray*}
Substituting the Taylor expansion  of $\beta_{t}(z)$ about $z=\xi(t)$ and using the fact that $\dot{\beta}_{t}(\xi(t))<\infty$, we deduce that $P(z,t)$ must have a simple pole at $z=\xi(t)$ of residue
\[\frac{-\dot{c}(t)}{\beta_{t}'(\xi(t))}.\]
Note that the fact that $c(t)$ is strictly increasing implies that $\dot{c}(t)>0$ and, also, the fact that $\beta_{t}$ maps an interval containing $\xi(t)$ to an interval containing $ \widetilde{\xi}(t)$ preserving orientation implies that $\beta_{t}'(\xi(t))>0$. Hence $P(z,t)$ has a simple pole at $z=\xi(t)$ of negative residue. By Lemma \ref{lem1}, this also implies that $P(z,t)$ has a simple pole of negative residue at $\phi_{t}(\xi(t))$ for each $\phi_{t}\in\Gamma_{t}$. In particular, the singularities of $P(z,t)$ at $\phi_{t}(\xi(t))$ are not branch point singularities. This implies that $P(z,t)$ can be extended to a meromorphic function in $\widehat{\mathbb{C}}\setminus\Lambda_{t}$

Finally, suppose that $p_{t}$ is the fixed point of a parabolic M\"{o}bius transformation in $\Gamma_{t}$. Note that the angular limit (i.e. the limit as $z\rightarrow p_{t}$ in a Stolz region at $p_{t}$ -- see \cite[p. 6]{MR1217706}) exists. Moreover, since $f_{t}$ maps $\mathbb{H}$ into itself, by the Julia-Wolff Lemma (see \cite[Proposition 4.13]{MR1217706}) the angular limit of $f_{t}'(z)$ as $z\rightarrow p_{t}$ also exists. This implies that the angular limit as $z\rightarrow p_{t}$ of $P(z,t)$ exists and in particular, the limit of $P(z,t)$ as $z\rightarrow p_{t}$ inside a fundamental region also exists.
\end{proof}
We will now construct a particular character-automorphic function (see Section 2) that will be important to us. This function also appears in \cite[Section 99]{Ford} and in \cite{Tsai2}.
\begin{lemma}\label{lem3}
Suppose that $\Gamma'$ is a Fuchsian group of the second kind with limit set $\Lambda'$ such that $\infty\not\in\Lambda'$. Let
\[\Psi(z)=\prod_{\phi\in\Gamma'}\left(\frac{z-\phi\circ \psi(c)}{z-\phi(c)}\right)^{2},\]
for some  $\psi \in\Gamma'$ and $c\in\mathbb{R}\setminus\Lambda'$. Then
\begin{enumerate}
 \item[(i)] $\Psi$ is well-defined and is a holomorphic function in $\widehat{\mathbb{C}}\setminus\Lambda'$ with no zeroes;
\item[(ii)]  $\Psi$ is character-automorphic with character
\[\chi(\rho)= \prod_{\phi\in\Gamma'}\frac{\rho'(\phi\circ\psi(c))}{\rho'(\phi(c))}\]
for $\rho\in\Gamma'$.
\item[(iii)]  if $\rho\in\Gamma'$ is a parabolic or elliptic M\"{o}bius transformation, then
\[\chi(\rho)=1;\]
\item[(iv)]  the limit of $\Psi(z)$ as $z$ tends to the fixed point of a parabolic element of $\Gamma'$ inside a fundamental region exists;
\item[(v)]   the function $\Psi$ does not depend on the choice of $c\in\mathbb{R}\setminus\Lambda'$ i.e. for $c,c'\in\mathbb{R}\setminus\Lambda'$,
\[\Psi_{t}(z)=\prod_{\phi\in\Gamma'}\left(\frac{z-\phi\circ \psi(c)}{z-\phi(c)}\right)^{2}=\prod_{\phi\in\Gamma'}\left(\frac{z-\phi\circ \psi(c')}{z-\phi(c')}\right)^{2}.\]
\end{enumerate}
\end{lemma}
\begin{proof}
(i): The fact that the infinite product in $\Psi $ converges follows from Lemma \ref{lempoin} (see \cite[Section 50]{Ford}).
Moreover,  for every term $(z-\phi(c))$ in the denominator the same term also appears in the numerator. Thus $\Psi $ has no poles or zeroes except at the accumulation points of $\phi(c)$ on $\mathbb{R}$ but these points are contained in the limit set $\Lambda'$ (see Section 2). Hence $\Psi$ has no zeroes or poles in $\widehat{\mathbb{C}}\setminus\Lambda'$.

\ \\
(ii): Note that for any $a\in\mathbb{C}$ and M\"{o}bius transformation $\phi$, we have
\begin{equation}(z-\phi(a))^2=\frac{\phi'(a)(\phi^{-1}(z)-a)^{2}}{(\phi^{-1})'(z)}.\label{leqleq}\end{equation}
Hence we can write
\[\Psi (z)=\prod_{\phi \in\Gamma'}\left(\frac{\phi '(\psi (c))}{\phi '(c)}\right) \left(\frac{\phi ^{-1}(z)-\psi (c)}{\phi ^{-1}(z)-c}\right)^{2 }.\]
This implies that for any $\rho \in\Gamma $, we have
\begin{eqnarray*}\Psi (\rho (z))&=&
\prod_{\phi \in\Gamma' }\left(\frac{\phi '(\psi (c))}{\phi '(c)}\right) \left(\frac{\phi ^{-1}\circ\rho (z)-\psi (c)}{\phi ^{-1}\circ\rho (z)-c}\right)^{2 }
\\&=&\left[ \prod_{\phi \in\Gamma' }\left(\frac{\phi '(\psi (c))(\rho ^{-1}\circ\phi )'(c)}{\phi '(c)(\rho ^{-1}\circ\phi )'(\psi (c))}\right) \right]\Psi (z)
\\ &=& \left[\prod_{\phi \in\Gamma' }\left(\frac{\rho '(\phi \circ \psi (c)))}{\rho '(\phi (c))}\right) \right]\Psi (z).
\end{eqnarray*}
Hence
\[\chi (\rho )=\prod_{\phi \in\Gamma' }\frac{\rho '(\phi \circ \psi (c)))}{\rho '(\phi (c))}.\]

\ \\
(iii): Now let $\Gamma_{\rho}$ be the subset of $\Gamma'$ consisting of all M\"{o}bius transformations of the form $\phi^{*}\circ\phi$ where $\phi^{*},\phi\in\Gamma'$ and $\phi^{*}\neq \rho$ such that every element in the left quotient space of $\Gamma'$ under the stabilizer of $\rho$ corresponds to a unique element in $\Gamma_{\rho}$. If $\rho\in\Gamma'$ is an elliptic transformation, then $\rho$ must be of finite order $M$ since $\Gamma'$ is discrete (see \cite[Theorem 5.3.11]{MR1393195}). Thus, $\Gamma$ is the disjoint union of $\rho^{k}(\Gamma_{\rho})$ for $k=0,\ldots,M-1$ and hence  we can write
\[\chi (\rho )=\prod_{\phi \in\Gamma_{\rho }}\prod_{k=0}^{M-1}\frac{\rho '(\rho^{k}\circ\phi \circ \psi (c))}{\rho '(\rho^{k}\circ \phi(c))}.\]
Then note that
\begin{eqnarray*}\prod_{k=0}^{M-1}\rho '(\rho^{k}\circ\phi \circ \psi (c))&=& \prod_{k=1}^{M}\frac{(\rho^{k+1}\circ\phi)'(\psi(c))}{(\rho^{k}\circ\phi)'(\psi(c))  }
\\ &=& \frac{(\rho^{M}\circ\phi)'(\psi(c))}{\phi'(\psi(c))  }=1 ,\end{eqnarray*}
where for the last equality, we have used the fact that $\rho$ is of order $M$. Similarly,
\[\prod_{k=1}^{M}\rho '(\rho^{k}\circ\phi(c))=1.\]
Thus
\[\chi (\rho )=1.\]
If $\rho$ is a parabolic M\"{o}bius transformation, then $\Gamma$ is the disjoint union of $\rho^{k}(\Gamma_{\rho})$ for $k\in\mathbb{Z}$ and hence
\[\chi (\rho )=\prod_{\phi \in\Gamma_{\rho }}\prod_{k=-\infty}^{\infty}\frac{\rho '(\rho^{k}\circ\phi \circ \psi (c))}{\rho '(\rho^{k}\circ \phi(c))}.\]
Then,
\begin{eqnarray*}\prod_{k=-\infty}^{\infty}\rho '(\rho^{k}\circ\phi \circ \psi (c))&=& \prod_{k=-\infty}^{\infty}\frac{(\rho^{k}\circ\phi)'(\psi(c))}{(\rho^{k-1}\circ\phi)'(\psi(c))  }
\\ &=& \lim_{M\rightarrow\infty}\frac{(\rho^{M}\circ\phi)'(\psi(c))}{\rho^{-M-1}\phi'(\psi(c))  }=1, \end{eqnarray*}
where, for the last equality, we have used the fact that $\rho^{M}$ and $\rho^{-M}$ both converge to the unique fixed point of $\rho$ (since $\rho$ is parabolic) as $M\rightarrow\infty$. Similarly,
\[\prod_{k=-\infty}^{\infty}\rho '(\rho^{k}\circ\phi(c))=1\]
and hence
\[\chi (\rho )=1.\]

\ \\
(iv): Suppose that $p\in\mathbb{R}$ is a fixed point of a parabolic element of $\Gamma'$. Then we can find a M\"{o}bius transformation $\zeta$ such that $\zeta(\mathbb{H})=\mathbb{H}$ and $\zeta(p)=\infty$. Then using (\ref{leqleq}), we can write
\[\Psi(\zeta(z))=\prod_{\phi\in\Gamma'}\frac{\zeta'(\zeta^{-1}\circ\phi\circ\psi(c))}{\zeta'(\zeta^{-1}\circ\phi(c))}\left(\frac{z-\zeta^{-1}\circ\phi\circ \psi(c)}{z-\zeta^{-1}\circ\phi(c)}\right)^{2}.\]
Then as $\mathrm{Im}[z]\rightarrow\infty$,
\[ \Psi(\zeta(z))\rightarrow \prod_{\phi\in\Gamma'}\frac{\zeta'(\zeta^{-1}\circ\phi\circ\psi(c))}{\zeta'(\zeta^{-1}\circ\phi(c))}\]
and this implies that the limit of $\Psi(z)$ as $z\rightarrow p$ inside a fundamental region exists.

\ \\
(v): Let
\[\Psi^{*}(z)=\left(\prod_{\phi\in\Gamma'}\frac{z-\phi\circ \psi(c')}{z-\phi(c')}\right)^{2}.\]
From the formula for the character in part (ii), we can see that the character of $\Psi$ and $\Psi^{*}$ only depend on the accumulation points of the sets
$\{\phi(c):\phi\in\Gamma'\}$ and $\{\phi(c'):\phi\in\Gamma'\}$ but this is exactly the limit set $\Lambda'$ (see Section 2). Hence the characters of $\Psi$ and $\Psi^{*}$ are identical and part (iv) implies that $\Psi(z)/\Psi^{*}(z)$ is a simple automorphic function with no zeroes or poles. By Lemma \ref{lemauto}(iii) it is constant. Moreover, since $\infty$ is an ordinary point of $\Gamma_{t}$, by considering the limit as $z\rightarrow\infty$, we deduce that $\Psi(z)/\Psi^{*}(z)=1$.
\end{proof}
We end this section with the main lemma which we will use to determine $P(z,t)$.
\begin{lemma}\label{lem4} Suppose that $\infty\not\in \Lambda_{t}$ for all $t\in[0,T]$ and there exists $c\in\mathbb{R}\setminus\Lambda_{t}$ for all $t\in[0,T]$.
Suppose that $\psi_{1,t},\ldots,\psi_{N,t}\in\Gamma_{t}$ satisfy
\[\psi_{k,t}=f_{t}^{-1}\circ\psi_{k}\circ f_{t}\]
for some $\psi_{k}\in\Gamma$ for each $k=1,\ldots,N$ and let $\delta_{1}(t),\ldots,\delta_{N}(t)$ be real-valued functions on $[0,T]$. For $t\in(0,T)$ such that $P(z,t)$ is defined, let
\[\Psi_{t}(z)=\prod_{k=1}^{N} \Psi_{k,t}(z)^{\delta_{k}(t)},\]
where, for $k=1,\ldots,N$,
\[\Psi_{k,t}(z)=\prod_{\phi_{t}\in\Gamma_{t}}\left[\frac{z-\phi_{t}\circ \psi_{k,t}(c)}{z-\phi_{t}(c)}\right]^{2}.\]
 Define
\[\Xi_{t}(z)=\frac{1}{2}\frac{\Psi_{t}'(z)}{\Psi_{t}(z)}P(z,t) -\sum_{k=1}^{N} \delta_{k}(t)\frac{\dot{\Psi}_{k,t}(z)}{\Psi_{k,t}(z)}.\]
Then $\Xi_{t}(z)$ is a meromorphic function in $\widehat{\mathbb{C}}\setminus\Lambda_{t}$ whose only poles are simple poles of real residue at $\phi_{t}(\xi(t))$ for each $\phi_{t}\in\Gamma_{t}$. Moreover, $\Xi_{t}(z)$ satisfies
\[\Xi_{t}(\rho_{t}(z))=\Xi_{t}(z)-\sum_{k=1}^{N}\delta_{k}(t)\frac{\frac{d}{dt}\left(\chi_{k,t}(\rho_{t})\right)}{\chi_{k,t}(\rho_{t})}\]
for any $\rho_{t}\in\Gamma_{t}$ where
\[\chi_{k,t}(\rho_{t})=\prod_{\phi_{t}\in\Gamma_{t}}\left(\frac{\rho_{t}'(\phi_{t}\circ\psi_{k,t}(c))}{\rho_{t}'(\phi_{t}(c))}\right)\]
for $k=1,\ldots,N$.
\end{lemma}
\begin{proof}
Firstly, by Lemma \ref{lem3}, $\Psi_{k,t}(z)$ satisfies
\[\Psi_{k,t}(\rho_{t}(z))=\chi_{k,t}(\rho_{t})\Psi_{k,t}(z)\]
for all $\rho_{t}\in\Gamma_{t}$. This implies that $\Psi_{t}(z)$ satisfies
\[\Psi_{t}(\rho_{t}(z))=\chi_{t}(\rho_{t})\Psi_{t}(z)\]
for all $\rho_{t}\in\Gamma_{t}$ where
\[\chi_{t}(\rho_{t})=\prod_{k=1}^{N}\chi_{k,t}(\rho_{t})^{\delta_{k}(t)}.\]
 This implies that for any $\rho_{t}\in\Gamma_{t}$,
\begin{eqnarray*}
\frac{\Psi_{t}'(z)}{\Psi_{t}(z)}&=&\rho_{t}'(z)\frac{\Psi_{t}'(\rho_{t}(z))}{\Psi_{t}(\rho_{t}(z))}=\rho_{t}'(z)\sum_{k=1}^{N}2\delta_{k}(t)\frac{\Psi_{k,t}'(\rho_{t}(z))}{\Psi_{k,t}(\rho_{t}(z))} \\
\frac{\dot{\Psi}_{k,t}(z)}{\Psi_{k,t}(z)}&=&\frac{\dot{\Psi}_{k,t}(\rho_{t}(z))}{\Psi_{k,t}(\rho_{t}(z))}+
\dot{\rho}_{t}(z)\frac{\Psi_{k,t}'(\rho_{t}(z))}{\Psi_{k,t}(\rho_{t}(z))} -  \delta_{k}(t)\frac{\frac{d}{dt}\left(\chi_{k,t}(\rho_{t})\right)}{\chi_{k,t}(\rho_{t})}
\end{eqnarray*}
Also, by Lemma \ref{lem1},
\[P(z,t)=\frac{1}{\rho_{t}'(z)}P(\rho_{t}(z),t)+\frac{\dot{\rho}_{t}(z)}{\rho_{t}'(z)}.\]
Substituting these into the formula for $\Xi_{t}(z)$, we get
\[\Xi_{t}(\rho_{t}(z))=\Xi_{t}(z)-\sum_{k=1}^{N}\delta_{k}(t)\frac{\frac{d}{dt}\left(\chi_{k,t}(\rho_{t})\right)}{\chi_{k,t}(\rho_{t})}.\]
By Lemmas \ref{lem2} and \ref{lem3}(i), we deduce that $\Xi_{t}(z)$ is a meromorphic function in $\widehat{\mathbb{C}}\setminus\Lambda$ whose only poles are simple poles of real residue at $\phi_{t}(\xi(t))$ for each $\phi_{t}\in\Gamma_{t}$.
\end{proof}
\begin{remark}
Lemma \ref{lem4} also holds if
\[\Psi_{k,t}(z) = \prod_{\phi_{t}\in\Gamma_{t}}\left[ \frac{z-a_{k}(t)}{z-b_{k}(t)}\right]^{2},\]
where $a_{k}(t)=f_{t}^{-1}(A_{k})$ and $b_{k}(t)=f_{t}^{-1}(B_{k})$ for
\[A_{k},B_{k}\in\mathbb{H}\setminus\pi^{-1}(\gamma(0,T]).\]
Note that in this case,  $\Xi_{t}(z)$ does not have poles at $a_{k}(t)$ and $b_{k}(t)$. This can be seen by considering the residue of  $\Xi_{t}$ at $z=\phi_{t}(a_{k}(t))$ for some $\phi_{t}\in\Gamma_{t}$ and  $k=1,\ldots,N$: this is
\[P(\phi_{t}(a_{k}(t)),t)-\phi_{t}'(a_{k}(t))\dot{a}_{k}(t)+\dot{\phi}_{t}(a_{k}(t))\]
and by Lemma \ref{lem1}, it is equal to zero. Similarly for $\phi_{t}(b_{k}(t))$.
\end{remark}
\section{The proof of Theorem \ref{genslitRS}}
Our strategy to prove Theorem \ref{genslitRS} is as follows: we first construct a function $\Upsilon_{t}$ that is a possible candidate for the function $\Xi_{t}$ given in Lemma \ref{lem4}. We will then show that if we choose $\psi_{1,t},\ldots,\psi_{N,t}$ and $\delta_{1}(t),\ldots,\delta_{N}(t)$  in Lemma \ref{lem4} appropriately, then $\Xi_{t}(z)$ will be equal to $\Upsilon_{t}(z)$ plus a function which only depends on $t$.
\begin{lemma}\label{lem5}
Suppose that $\Gamma'$ is a Fuchsian group of the second kind with limit set $\Lambda'$ and let $C\in\mathbb{R}$ and $c\in\mathbb{R}\setminus\Lambda'$. Let
\[\Upsilon_{\Gamma'}(z)=C+\sum_{\phi\in\Gamma'}\frac{1}{\phi(z)}-\frac{1}{\phi(c)}.\]
Then
\begin{enumerate}
\item[(i)] $\Upsilon_{\Gamma'}$ is a meromorphic function in $\widehat{\mathbb{C}}\setminus\Lambda'$ whose only poles are simple poles of real residue at $\phi(0)$ for each $\phi\in\Gamma'$;
\item[(ii)] $\Upsilon_{\Gamma'}$ satisfies
\[\Upsilon_{\Gamma'}(\rho(z))=\Upsilon_{\Gamma'}(z)+\left[\sum_{\phi\in\Gamma'}\frac{1}{\phi\circ\rho(c)}-\frac{1}{\phi(c)}\right].\]
for any $\rho\in\Gamma'$;
\item[(iii)] if $\rho\in\Gamma'$ is an elliptic or parabolic M\"{o}bius transformation, then
\[\Upsilon_{\Gamma'}(\rho(z))=\Upsilon_{\Gamma'}(z);\]
\item[(iv)] the limit of $\Upsilon_{\Gamma'}(z)$ as $z$ tends to a fixed point of a parabolic element of $\Gamma'$ inside a fundamental region exists;
\item[(v)]  up to a different constant $C$, $\Upsilon_{\Gamma'}(z)$ does not depend on the choice of $c\in\mathbb{R}\setminus\Lambda'$ i.e. for $c'\in\mathbb{R}\setminus\Lambda'$,
\[\Upsilon_{\Gamma'}(z)=C'+\sum_{\phi\in\Gamma'}\frac{1}{\phi(z)}-\frac{1}{\phi(c')}\]
for some $C'\in\mathbb{R}$.
\end{enumerate}
\end{lemma}
\begin{proof}
(i): Note that
\[\frac{1}{\phi(z)}-\frac{1}{\phi(c)}= -\frac{\phi(z)-\phi(c)}{\phi(z)\phi(c)}.\]
Hence the locally uniform convergence of the sum in the definition of $\Upsilon_{\Gamma'}$ follows from Lemma \ref{lempoin} as in the proof of Lemma \ref{lem3}(i). Hence $\Upsilon_{\Gamma'}$ is a meromorphic function in $\widehat{\mathbb{C}}\setminus\Lambda'$ whose only pole is a simple pole at $\phi(0)$ for each $\phi\in\Gamma'$.

\ \\
(ii): For $\rho\in\Gamma'$,
\begin{eqnarray*}\Upsilon_{\Gamma'}(\rho(z))&=&\sum_{\phi\in\Gamma'}\frac{1}{\phi\circ\rho(z)}-\frac{1}{\phi(c)}\\
&=&\sum_{\phi\in\Gamma'}\left(\frac{1}{\phi\circ\rho(z)}-\frac{1}{\phi\circ\rho(c)}\right)+\left(\frac{1}{\phi\circ\rho(c)}-\frac{1}{\phi(c)}\right)\\
 &=& \Upsilon_{\Gamma'}(z)+\left[\sum_{\phi\in\Gamma'}\frac{1}{\phi\circ\rho(c)}-\frac{1}{\phi(c)}\right].
\end{eqnarray*}

\ \\
(iii): Let
\[K(\rho)=\sum_{\phi\in\Gamma'}\frac{1}{\phi\circ\rho(c)}-\frac{1}{\phi(c)}.\]
Then the proof that $K(\rho)=0$ whenever $\rho$ is an elliptic or parabolic M\"{o}bius transformation in $\Gamma'$ is similar to proof of Lemma \ref{lem3}(iii). We leave this to the reader.

\ \\
(iv): The proof of this is almost identical to the proof of Lemma \ref{lem3}(iv) so we also omit this proof.

\ \\
(v): Finally, if
\[\Upsilon^{*}(z)=\sum_{\phi\in\Gamma'}\frac{1}{\phi(z)}-\frac{1}{\phi(c')},\]
then
\[\Upsilon_{\Gamma'}(z)-\Upsilon^{*}(z)=\sum_{\phi\in\Gamma'}\frac{1}{\phi(c)}-\frac{1}{\phi(c')}=\text{constant}.\]
\end{proof}
As in the previous section, we will suppose that $\gamma:[0,T]\rightarrow R\cup \partial R$ is a simple curve such that $\gamma(0,T]\subset R$ and $\gamma(0)\in\partial R$ and $f_{t}$ is a Loewner chain corresponding to the curve $\gamma$ on $R$ so that each $f_{t}$ is a conformal map of $\mathbb{H}$ onto $\mathbb{H}\setminus\pi^{-1}(\gamma(0,t])$. We can write
\[\pi^{-1}(\gamma(0,t])=\bigcup_{\phi\in\Gamma}\phi( \widetilde{\gamma}(0,t]),\]
where $\widetilde{\gamma}:[0,T]\rightarrow \overline{\mathbb{H}}$ is a simple curve with $\widetilde{\gamma}(0)\in\mathbb{R}$ and $\widetilde{\gamma}(0,T]\subset\mathbb{H}$. Let $\xi(t)=f_{t}^{-1}(\widetilde{\gamma}(t))$.

Let $\Gamma'=\Gamma_{t}^{\xi}$ where $\Gamma_{t}^{\xi}$ is obtained by conjugating each M\"{o}bius transformation in $\Gamma_{t}$ with the map $z\mapsto z+\xi(t)$ i.e.
\[\Gamma_{t}^{\xi}=\{z\mapsto \phi_{t}(z+\xi(t))-\xi(t) :\phi_{t}\in\Gamma_{t}\}.\]
With $\Upsilon_{\Gamma_{t}^{\xi}}(z)$ as defined in Lemma \ref{lem5} (with $C=0$), we define
\begin{eqnarray*}\Upsilon_{t}(z)&=&\Upsilon_{\Gamma_{t}^{\xi}}(z-\xi(t))
\\ &=& \sum_{\phi_{t}\in\Gamma_{t}}\frac{1}{\phi_{t}(z)-\xi(t)}-\frac{1}{\phi_{t}(c)-\xi(t)}.
\end{eqnarray*}
By Lemma \ref{lem5}(i), $\Upsilon_{t}(z)$ is a meromorphic function in $\widehat{\mathbb{C}}\setminus\Lambda_{t}$ whose only poles are simple poles of real residue at $\phi_{t}(\xi(t))$ for each $\phi_{t}\in\Gamma_{t}$. Also, by Lemma \ref{lem5}(ii),
\[\Upsilon_{t}(\rho_{t}(z))=\Upsilon_{t}(z)+ \left[\sum_{\phi_{t}\in\Gamma_{t}}\frac{1}{\phi_{t}\circ\rho_{t}(c)-\xi(t)}-\frac{1}{\phi_{t}(c)-\xi(t)}\right]\]
for all $\rho_{t}\in\Gamma_{t}$.

We want to construct a function $\Xi_{t}(z)$ (as in Lemma \ref{lem4}) such that
\[\Xi_{t}(\rho_{t}(z))-\Upsilon_{t}(\rho_{t}(z))=\Xi_{t}(z)-\Upsilon_{t}(z)\]
for all $\rho_{t}\in\Gamma_{t}$. Note that, by Lemmas \ref{lem4} and \ref{lem5}(i), $\Xi_{t}(z)$ and $\Upsilon_{t}(z)$ are both meromorphic functions on $\mathbb{H}\setminus\Lambda_{t}$ whose only poles are single poles of real residue at the $\Gamma_{t}$-orbit of $\xi(t)$. Hence $\Theta_{t}(z)=\Xi_{t}(z)-\Upsilon_{t}(z)$ is an automorphic function with exactly one pole in any fundamental region of $\Gamma_{t}$. Then Lemma \ref{lemauto}(iv) will imply that $\Theta_{t}(z)$ depends only on $t$.

Note that by Lemmas \ref{lem4} and \ref{lem5}(ii), $\Xi_{t}(z)-\Upsilon_{t}(z)$ satisfies
\[\Xi_{t}(\rho_{t}(z))-\Upsilon_{t}(\rho_{t}(z))=\Xi_{t}(z)-\Upsilon_{t}(z) -(J_{t}(\rho_{t})+K_{t}(\rho_{t})),\] where
\begin{eqnarray}\label{7th}
J_{t}(\rho_{t})&=&\sum_{k=1}^{N}\delta_{k}(t)\frac{\frac{d}{dt}\left(\chi_{k,t}(\rho_{t})\right)}{\chi_{k,t}(\rho_{t})},\\ \label{8th}
K_{t}(\rho_{t})&=&\sum_{\phi_{t}\in\Gamma_{t}}\frac{1}{\phi_{t}\circ\rho_{t}(c)-\xi(t)}-\frac{1}{\phi_{t}(c)-\xi(t)},
\end{eqnarray}
where $\chi_{k,t}(\rho_{t})$ for $k=1,\ldots N$ is defined in Lemma \ref{lem4}. Hence we want to show that $J_{t}(\rho_{t})+K_{t}(\rho_{t})=0$ for all $\rho_{t}\in\Gamma_{t}$. By Lemmas \ref{lem3}(iii) and \ref{lem5}(iii),  when $\rho_{t}$ is an elliptic or parabolic M\"{o}bius transformation $J_{t}(\rho_{t})=K_{t}(\rho_{t})=0$ and hence
$J_{t}(\rho_{t})+K_{t}(\rho_{t})=0$.

Also note that $J_{t}$ and $K_{t}$ satisfy
\begin{eqnarray*}
J_{t}(\rho_{1,t}\circ\rho_{2,t})&=& J_{t}(\rho_{1,t})+J_{t}(\rho_{2,t})\\
K_{t}(\rho_{1,t}\circ\rho_{2,t})&=& K_{t}(\rho_{1,t})+K_{t}(\rho_{2,t})
\end{eqnarray*}
for any $\rho_{1,t},\rho_{2,t}\in\Gamma_{t}$. Hence $J_{t}$ and $K_{t}$ are uniquely determined by their values at the hyperbolic generators of $\Gamma_{t}$. For $t\in(0,T]$, we suppose that $\Gamma_{t}$ is generated freely by hyperbolic M\"{o}bius transformations $\phi_{1,t},\ldots,\phi_{n,t}$ and possibly some other elliptic or parabolic M\"{o}bius transformations. We can assume that
\[\phi_{k,t}=f_{t}^{-1}\circ\phi_{k}\circ f_{t}\]
for some $\phi_{k}\in\Gamma$ and hence $\phi_{k,t}(z)$ is differentiable almost everywhere with respect to $t$. We wish to construct $\Psi_{t}(z)$ such that $J_{t}(\phi_{k,t})+K_{t}(\phi_{k,t})=0$ for each $k=1,\ldots,n$.
\begin{lemma}\label{lem7}
Suppose that $\infty\not\in \Lambda_{t}$ for all $t\in[0,T]$ and there exists $c\in\mathbb{R}\setminus\Lambda_{t}$ for all $t\in[0,T]$. Suppose that $R_{t}:\Gamma_{t}\rightarrow \mathbb{R}$ satisfies
\[R_{t}(\rho_{1,t}\circ \rho_{2,t})=R_{t}(\rho_{1,t})+R_{t}(\rho_{2,t}) \text{ for all } \rho_{1,t},\rho_{2,t}\in\Gamma_{t}\]
and $R(\rho_{t})=0$ for elliptic or parabolic M\"{o}bius transformations $\rho_{t}\in\Gamma_{t}$.
Then, for $t\in(0,T)$ such that $P(z,t)$ is defined, we can find $\delta_{1}(t),\ldots,\delta_{n}(t)\in\mathbb{R}$ such that if
\[\Psi_{k,t}(z)=\left(\prod_{\phi_{t}\in\Gamma_{t}}\frac{z-\phi_{t}\circ \phi_{k,t}(c) }{z-\phi_{t}(c)}\right)^{2}\]
and
\[\Psi_{t}(z)=\prod_{k=1}^{n}\Psi_{k,t}(z)^{\delta_{k}(t)},\]
then $\Xi_{t}(z)$ as defined in Lemma \ref{lem4} satisfies
\[\Xi_{t}(\rho_{t}(z))=\Xi_{t}(z)+R_{t}(\rho_{t}) \text{ for any }\rho_{t}\in\Gamma_{t}.\]
\end{lemma}
\begin{proof}
First note that if
\[\Psi_{t}(z)=\prod_{k=1}^{n}\left(\prod_{\phi_{t}\in\Gamma_{t}}\frac{z-\phi_{t}\circ \phi_{k,t}(c) }{z-\phi_{t}(c)}\right)^{2\delta_{k}(t)},\]
then, by Lemma \ref{lem4},
\[\Xi_{t}(\rho_{t}(z))=\Xi_{t}(z)-\sum_{k=1}^{n}\delta_{k}(t)\frac{\frac{d}{dt}\left(\chi_{k,t}(\rho_{t})\right)}{\chi_{k,t}(\rho_{t})} \text{ for any }\rho_{t}\in\Gamma_{t}\]
and, by Lemma \ref{lem3}(iii), if $\rho_{t}$ is a parabolic or elliptic M\"{o}bius transformation,
\[\sum_{k=1}^{n}\delta_{k}(t)\frac{\frac{d}{dt}\left(\chi_{k,t}(\rho_{t})\right)}{\chi_{k,t}(\rho_{t})}=0.\]
Let
\[X_{k}(t)=\left(\frac{\frac{d}{dt}\left(\chi_{k,t}(\phi_{1,t})\right)}{\chi_{k,t}(\phi_{1,t})},\ldots,\frac{\frac{d}{dt}\left(\chi_{k,t}(\phi_{n,t})\right)}{\chi_{k,t}(\phi_{n,t})}\right)\]
for $k=1,\ldots, n$. Since the value of $R_{t}(\rho_{t})$ is determined by the values of $R_{t}(\phi_{k,t})$ for $k=1,\ldots, n$, we only need to show that we can find $\delta_{1}(t),\ldots,\delta_{n}(t)\in\mathbb{R}$ such that
\[\delta_{1}(t)X_{1}(t)+\cdots+\delta_{n}(t)X_{n}(t)=(-R_{t}(\phi_{1,t}),\ldots,-R_{t}(\phi_{n,t})).\]
This follows if $X_{1}(t),\ldots X_{n}(t)$ form a basis of $\mathbb{R}^{n}$. Thus it is enough to show that $X_{1}(t),\ldots, X_{n}(t)$ are linearly independent. So for fixed $t$, suppose that we can find $\alpha_{1},\ldots,\alpha_{n}\in\mathbb{R}$ such that
\[\alpha_{1}X_{1}(t)+\cdots+\alpha_{n}X_{n}(t)=(0,\ldots,0).\]
This implies that if we define
\[\Psi_{t}^{*}(z)=\prod_{k=1}^{n}\Psi_{k,t}(z)^{\alpha_{k}},\]
then $\Xi_{t}^{*}(z)$ defined from $\Psi_{k,t}(z)$ and $\alpha_{k}$ for $k=1,\ldots,n$ using Lemma \ref{lem4} satisfies
\[\Xi_{t}^{*}(\rho_{t}(z))=\Xi_{t}^{*}(z)\]
for every $\rho_{t}\in\Gamma_{t}$ Then this fact and Lemmas \ref{lem2} and \ref{lem3}(iv) imply that $\Xi_{t}^{*}(z)$ is a simple automorphic function whose only poles are simple poles of real residue at the $\Gamma_{t}$-orbit of $\xi(t)$. Lemma \ref{lemauto}(iv) implies that the residue at each simple pole of $\Xi_{t}^{*}(z)$ must be zero i.e. $\Xi_{t}^{*}(z)$ has no poles. Since $P(z,t)$ has poles at the $\Gamma_{t}$-orbit of $\xi(t)$ by Lemma \ref{lem2}, we deduce from the formula for $\Xi_{t}^{*}(z)$,  that
\[\frac{(\Psi_{t}^{*})'(z)}{\Psi_{t}^{*}(z)}=0.\]
This implies that $\Psi_{t}^{*}(z)$ is constant. Furthermore, by considering the behaviour as $z\rightarrow\infty$ of $\Psi_{t}^{*}(z)$, we must have $\Psi_{t}^{*}(z)=1$. We will show that this implies that
\[\alpha_{1}=\cdots=\alpha_{n}=0.\]
First suppose that
 \[\alpha_{1}=\frac{a_{1}}{b_{1}},\ldots,\alpha_{n}=\frac{a_{n}}{b_{n}},\]
where $a_{k},b_{k}\in\mathbb{Z}$ for $k=1,\ldots,n$.
Then using Lemma \ref{lem3}(v), we can write $\Psi_{t}(z)$ as
\[\Psi_{t}(z)=\left[\prod_{\phi_{t}\in\Gamma_{t}}\frac{z-\phi_{t}\circ \phi_{t}^{*}(c) }{z-\phi_{t}(c)}\right]^{2b},\]
where $\phi_{t}^{*}=\phi_{t}^{a_{1}}\circ\cdots\circ\phi_{t}^{a_{n}}$ and $b=b_{1}\cdots b_{n}$. Hence
\[\prod_{\phi_{t}\in\Gamma_{t}}\frac{z-\phi_{t}\circ \phi_{t}^{*}(c) }{z-\phi_{t}(c)}=1.\]
This implies that $\phi^{*}$ is the identity and, using the fact that $\phi_{1,t},\ldots, \phi_{n,t}$ are free,  this implies that $a_{k}=0$ for $k=1,\ldots,n$. Thus
\[\alpha_{1}=\cdots=\alpha_{n}=0.\]

Now suppose that $\alpha_{1},\ldots,\alpha_{n}$ are not all rational. Then by approximating $\alpha_{1},\ldots,\alpha_{n}$ by rational numbers and following a similar argument, we deduce that we can find $a_{1,j},\ldots,a_{n,j}\in\mathbb{Z}$ such that
\[\phi_{1,t}^{a_{1,j}}\circ\cdots\circ\phi_{n,t}^{a_{n,j}}\rightarrow \mathrm{Id} \text{ locally uniformly, as } j\rightarrow\infty\]
where $\mathrm{Id}$ is the identity mapping. By discreteness, this implies that for sufficiently large $j$, we must have
\[\phi_{1,t}^{a_{1,j}}\circ\cdots\circ\phi_{n,t}^{a_{n,j}}=\mathrm{Id}.\]
As in the case where $\alpha_{1},\cdots,\alpha_{n}$ are rational, this implies that
\[\alpha_{1}=\cdots=\alpha_{n}=0.\]
\end{proof}
\begin{remark}\
  Lemma \ref{lem7} is very similar to the Lemma 3.7 in \cite{Tsai2}.
\end{remark}
This lemma allows us to construct the desired function $\Xi_{t}(z)$.
\begin{lemma}\label{lem8} Suppose that $\infty\not\in \Lambda_{t}$ for all $t\in[0,T]$ and there exists $c\in\mathbb{R}\setminus\Lambda_{t}$ for all $t\in[0,T]$.
For $t\in(0,T)$ such that $P(z,t)$ is defined, let $\delta_{1}(t),\ldots,\delta_{n}(t)$ be the solution of the system of linear equations
\[\delta_{1}(t)\frac{\frac{d}{dt}\left(\chi_{1,t}(\phi_{k,t})\right)}{\chi_{1,t}(\phi_{k,t})}+\cdots+\delta_{n}(t)\frac{\frac{d}{dt}\left(\chi_{n,t}(\phi_{k,t})\right)}{\chi_{n,t}(\phi_{k,t})}
=\sum_{\phi_{t}\in\Gamma_{t}}\frac{1}{\phi_{t}(c)-\xi(t)}-\frac{1}{\phi_{t}\circ\phi_{k,t}(c)-\xi(t)}
\]
for $k=1,\ldots,n$. Note that this system of linear equations always has unique solution. If
\[\Psi_{k,t}(z)=\prod_{\phi_{t}\in\Gamma_{t}}\left(\frac{z-\phi_{t}\circ \phi_{k,t}(c) }{z-\phi_{t}(c)}\right)^{2}\]
and we define $\Xi_{t}(z)$ as in Lemma \ref{lem4}, then $\Xi_{t}(z)-\Upsilon_{t}(z)$ is automorphic.
\end{lemma}
\begin{proof}
Lemma \ref{lem7} guarantees that this system of equations always has solution. Then note that (\ref{7th}), (\ref{8th}) implies that
\[\Xi_{t}(\phi_{k,t}(z))-\Upsilon_{t}(\phi_{k,t}(z))=\Xi_{t}(z)-\Upsilon_{t}(z)\]
for all $k=1,\ldots,n$. This implies that
\[\Xi_{t}(\rho_{t}(z))-\Upsilon_{t}(\rho_{t}(z))=\Xi_{t}(z)-\Upsilon_{t}(z)\]
for all $\rho_{t}\in\Gamma_{t}$.
\end{proof}
We will now use the preceding lemmas to prove Theorem \ref{genslitRS}.
\begin{proof}[Proof of Theorem \ref{genslitRS}]
By Lemma \ref{lem8}, we can find $\delta_{1}(t),\ldots,\delta_{n}(t)$ such that $\Theta_{t}(z)=\Xi_{t}(z)-\Upsilon_{t}(z)$ is an automorphic function. Moreover, by Lemmas \ref{lem3}(iv) and \ref{lem5}(iv), the limit of $\Theta_{t}(z)$ as $z$ tends to a fixed point of a parabolic element of $\Gamma_{t}$ exists. Hence $\Theta_{t}(z)$ is a simple automorphic function. Then note that the only possible poles of $\Theta_{t}(z)$ are simple poles at  $\phi_{t}(\xi(t))$ for each $\phi_{t}\in\Gamma_{t}$. Thus Lemma \ref{lemauto}(iv) implies that $\Theta_{t}(z)=\lambda^{*}(t)$ for some $\lambda^{*}(t)\in\mathbb{R}$ and thus $\Xi_{t}(z)=\Upsilon_{t}(z)+\lambda^{*}(t)$. This implies that
\[P(z,t)=\frac{\Psi_{t}(z)}{\Psi_{t}'(z)}\left[\sum_{k=1}^{n}2\delta_{k}(t)\frac{\dot{\Psi}_{k,t}(z)}{\Psi_{k,t}(z)} + 2\Upsilon_{t}(z)+\lambda^{*}(t)\right],\]
where
\[\Psi_{k,t}(z)=\prod_{\phi_{t}\in\Gamma_{t}} \left(\frac{z-\phi_{t}\circ\phi_{k,t}(c)}{z-\phi_{t}(c)}\right)^{2}\]
and
\[\Psi_{t}(z)=\prod_{k=1}^{n}\Psi_{k,t}(z)^{\delta_{k}(t)}.\]
Note that, as in the proof of Lemma \ref{lem3}(ii), we can write
\[\Psi_{k,t}(z)=\prod_{\phi_{t}\in\Gamma_{t}}R_{k,t}(\phi_{t})\left(\frac{\phi_{t}(z)-\phi_{k,t}(c)}{\phi_{t}(z)-c}\right)^{2},\]
where
\[R_{k,t}(\phi_{t})=\prod_{\phi_{t}\in\Gamma_{t}}\frac{(\phi_{t}^{-1})'(\phi_{k,t}(c))}{(\phi_{t}^{-1})'(c)}\]
for $k=1,\ldots,n$. This implies that
\begin{eqnarray*}
\frac{\Psi_{t}'(z)}{\Psi_{t}(z)}&=&\sum_{k=1}^{n}\sum_{\phi_{t}\in\Gamma_{t}} 2\delta_{k}(t)\left[\frac{\phi_{t}'(z)}{\phi_{t}(z)-\phi_{k,t}(c)}-\frac{\phi_{t}'(z)}{\phi_{t}(z)-c}\right] \\
\frac{\dot{\Psi}_{k,t}(z)}{\Psi_{k,t}(z)}&=&\sum_{\phi_{t}\in\Gamma_{t}}\left[\frac{\dot{\phi}_{t}(z)-\dot{\phi}_{k,t}(c)}{\phi_{t}(z)-\phi_{k,t}(c)}-\frac{\dot{\phi}_{t}(z)}{\phi_{t}(z)-c} + \frac{\frac{d}{dt}\left[R_{k,t}(\phi_{t})\right]}{R_{k,t}(\phi_{t})}\right].
\end{eqnarray*}
We write
\[P(z,t)=\frac{\lambda^{*}(t)+P_{1}^{*}(z,t)}{P_{2}(z,t)},\]
where
\[P_{1}^{*}(z,t)=\Upsilon_{t}(z)+\sum_{k=1}^{n}\delta_{k}(t)\frac{\dot{\Psi}_{k,t}(z)}{\Psi_{k,t}(z)}  \text{ and } P_{2}(z,t)=\frac{1}{2}\frac{\Psi_{t}'(z)}{\Psi_{t}(z)}.\]
Note that \[P_{2}(\rho_{t}(z),t)= \frac{1}{\rho_{t}'(z)}P_{2}(z,t)\] and also, by Lemma \ref{lem1}, \[P(\rho_{t}(z),t)=\rho_{t}'(z)P(z,t)-\dot{\rho}_{t}(z)\]
 for all $\rho_{t}\in\Gamma_{t}$. This implies that
\begin{equation}P_{1}^{*}(\rho_{t}(z),t)=P_{1}^{*}(z,t)-\dot{\rho}_{t}(z)P_{2}(z,t)\label{hemp}\end{equation}
for all $\rho_{t}\in\Gamma_{t}$. However,
\small\begin{eqnarray*}
P_{1}^{*}(z,t)&=&\sum_{\phi_{t}\in\Gamma_{t}}\left[\frac{1}{\phi_{t}(z)-\xi(t)} -\frac{1}{\phi_{t}(c)-\xi(t)} + \sum_{k=1}^{n} \delta_{k}(t)\left(\frac{\dot{\phi}_{t}(z)-\dot{\phi}_{k,t}(c)}{\phi_{t}(z)-\phi_{k,t}(c)}-\frac{\dot{\phi}_{t}(z)}{\phi_{t}(z)-c} +\frac{\frac{d}{dt}\left[R_{k,t}(\phi_{t})\right]}{R_{k,t}(\phi_{t})} \right)  \right],
\\ P_{2}(z,t)&=&
\sum_{\phi_{t}\in\Gamma_{t}}\sum_{k=1}^{n}2\delta_{k}(t)\left(\frac{\phi_{t}'(z)}{\phi_{t}(z)-\phi_{k,t}(c)}-\frac{\phi_{t}'(z)}{\phi_{t}(z)-c}\right).\end{eqnarray*}\normalsize
Substituting this into (\ref{hemp}), we deduce that we must have
\[\left[\sum_{\phi_{t}\in\Gamma_{t}}\left(\sum_{k=1}^{n}\frac{\frac{d}{dt}\left[R_{k,t}(\phi_{t})\right]}{R_{k,t}(\phi_{t})} \right)-\frac{1}{\phi_{t}(c)-\xi(t)}\right]<\infty\]
and hence
\[P_{1}(z,t)=\sum_{\phi_{t}\in\Gamma_{t}}\left[\frac{1}{\phi_{t}(z)-\xi(t)} + \sum_{k=1}^{n} \delta_{k}(t)\left(\frac{\dot{\phi}_{t}(z)-\dot{\phi}_{k,t}(c)}{\phi_{t}(z)-\phi_{k,t}(c)}-\frac{\dot{\phi}_{t}(z)}{\phi_{t}(z)-c} \right)  \right]\]
converges. Hence for some real-valued function $\lambda(t)$, \[\lambda^{*}(t)+P^{*}_{1}(z,t)=\lambda(t)+P_{1}(z,t)\] and thus
\[P(z,t)=\frac{\lambda(t)+P_{1}(z,t)}{P_{2}(z,t)}.\]
\end{proof}
\begin{remark}\
\begin{enumerate}
\item[(i)] Note that by Lemmas \ref{lem3}(v) and \ref{lem5}(v), choosing a different value for $c$ only changes the value of $\lambda(t)$.
\item[(ii)] Similarly, choosing a different lift $\tilde{\gamma}$ of $\gamma$ (and hence a different but $\Gamma_{t}$-equivalent $\xi(t)$) also changes the value of $\lambda(t)$.
\item[(iii)] Let $\{\psi_{t}\}_{t\in[0,T]}$ be a family of M\"{o}bius transformations preserving $\mathbb{H}$ such that $\psi_{t}(z)$ is differentiable for almost all $t$. By considering $\widetilde{f}_{t}=f_{t}\circ \psi_{t}$, we can see that we get the same formula for $P(z,t)$ when $\infty\in\Lambda_{t}$.
\item[(iv)] Similarly, by considering $\widetilde{f}_{t}=f_{t}\circ \psi_{t}$, we deduce that the function $\lambda(t)$ depends only on the normalization of the Loewner chain $\{f_{t}\}$ and in particular, we can choose a normalization for the functions $f_{t}$ such that  $\lambda(t)$ is any continuous real-valued function.
\item[(v)] If we reparameterize by $t\mapsto r(t)$, then $\widehat{f}_{t}=f_{r(t)}$ is a Loewner chain corresponding to $\widehat{\gamma}$ where $\widehat{\gamma}(t)=\gamma(t)$. Then
\[\dot{\widehat{f}}_{t}(z)=\widehat{f}_{t}'(z)\widehat{P}(z,t)\]
where
\[\widehat{P}(z,t)=\dot{r}(t)P(z,t).\]
Hence, in particular, we can reparameterize such that
\normalsize\[P(z,t)=\frac{\lambda(t)+\sum_{\phi_{t}\in\Gamma_{t}}\left[\frac{2\sigma_{t}}{\phi_{t}(z)-\xi(t)} + \sum_{k=1}^{n} \delta_{k}(t)\left(\frac{\dot{\phi}_{t}(z)-\dot{\phi}_{k,t}(c)}{\phi_{t}(z)-\phi_{k,t}(c)}-\frac{\dot{\phi}_{t}(z)}{\phi_{t}(z)-c} \right)  \right]}{\sum_{\phi_{t}\in\Gamma_{t}}\sum_{k=1}^{n}\delta_{k}(t)\left(\frac{\phi_{t}'(z)}{\phi_{t}(z)-\phi_{k,t}(c)}-\frac{\phi_{t}'(z)}{\phi_{t}(z)-c}\right)}\]\normalsize
where \[\sigma_{t}=-\sum_{\phi_{t}\in\Gamma_{t}}\sum_{k=1}^{n}\delta_{k}(t)\left(\frac{\phi_{t}'(\xi(t))}{\phi_{t}(\xi(t))-\phi_{k,t}(c)}-\frac{\phi_{t}'(\xi(t))}{\phi_{t}(\xi(t))-c}\right).\]
In particular, the residue at $z=\xi(t)$ of $P(z,t)$ is $-2$.
\item[(vi)] We can also apply Theorem \ref{genslitRS} to the case where $\gamma$ is not a simple curve: in this case, we suppose that the functions $f_{t}$ map $\mathbb{H}$ conformally onto the unique $\Gamma$-invariant connected component of $\mathbb{H}\setminus\pi^{-1}(\gamma(0,t])$. Note that $\xi(t)$ is then, in general, not continuous. On the other hand if $\gamma$ is non-crossing, then $\xi(t)$ is still continuous.
\end{enumerate}
\end{remark}
\section{A special case}
We will consider the special case where $\Gamma=\langle z\mapsto e^{\tau_{0}}z\rangle$ for some $\tau_{0}>0$. Note that $R=\mathbb{H} /\Gamma$ is
conformally equivalent to the annulus \[\mathbb{A}_{2\pi^{2}/\tau_{0}}=\{z: 1 < |z| < e^{2\pi^{2}/\tau_{0}}\}\] via the conformal map $z\mapsto \exp(-i 2\pi\log(z)/\tau_{0})$.

Now let $\gamma:[0,T]\rightarrow R\cup \partial R$ be a simple curve such that $\gamma(0,T]\subset R$ and $\gamma(0)\in\partial R$. We define $\widetilde{\gamma}$ to be a lift of $\gamma$ under the quotient map that starts from a point on $\mathbb{R}$. Let $f_{t}$ be a conformal map of $\mathbb{H}$ onto \[H_{t}=\mathbb{H}\setminus \bigcup_{k=-\infty}^{\infty}\{e^{k\tau_{0}s}:s\in\widetilde{\gamma}(0,t]\}.\]
It turns out that if we normalize $f_{t}$ such that $f_{t}(0)=0$ and $f_{t}(\infty)=\infty$, then we are able to determine the Fuchsian groups $\Gamma_{t}$ beforehand: the mapping $z\mapsto f_{t}^{-1}(e^{\tau_{0}(t)}f_{t}(z))$ generates $\Gamma_{t}$ and is a M\"{o}bius transformation preserving $\mathbb{H}$ that fixes $0$ and $\infty$; this implies that \[f_{t}^{-1}(e^{\tau_{0}(t)}f_{t}(z))=e^{\tau(t)z}\]
for some $\tau(t)\in\mathbb{R}$ and hence
\[\Gamma_{t}=\langle z\mapsto e^{\tau(t)}z\rangle.\]
Thus $R\setminus \gamma(0,t]$ is conformally equivalent to $\mathbb{H}/\Gamma_{t}$. Moreover, since the modulus of $R\setminus \gamma(0,t]$ is strictly decreasing, this implies that $\tau(t)$ is strictly increasing. Also, it is easy to see that with this normalization, $f_{t}(z)$ is differentiable with respect to $t$ in $(0,T)$. Let $\xi(t)=f_{t}^{-1}(\widetilde{\gamma}(t))$. We want to find
\[P(z,t)=\frac{\dot{f}_{t}(z)}{f_{t}'(z)}.\]

Although in this case, we have $\infty\in\Lambda_{t}$ for all $t\in[0,T]$, we can still apply the methods in  Lemma \ref{lem4} and the proof of Theorem \ref{genslitRS} with the function
\[\Psi_{t}(z)=\left[\prod_{k=-\infty}^{\infty}a_{k}(t)\frac{z-e^{(k+1)\tau(t)}c}{z-e^{k\tau(t)}c}\right]^{\delta(t)},\]
where
\[a_{k}(t)=\left\{\begin{array}{ll} e^{-\tau(t)} & \text{ for } k>0, \\ 1 & \text{ for } k\leq  0,\end{array} \right.\]
and $c\not\in\mathbb{R}\setminus\{0\}$. We note that
\begin{eqnarray*}\Psi_{t}(z)&=&\lim_{N\rightarrow\infty} \left[\prod_{k=-N}^{N}a_{k}(t)\frac{z-e^{(k+1)\tau(t)}c}{z-e^{k\tau(t)}c}\right]^{\delta(t)}
\\ &=&\left[\lim_{N\rightarrow\infty} e^{-(N+1)\tau(t)}\frac{z-e^{(N+1)\tau(t)}c}{z-e^{-N\tau(t)}c}\right]^{\delta(t)}
\\ &=&\left(\frac{c}{z}\right)^{\delta(t)}.
\end{eqnarray*}
Hence
\[\Psi_{t}(e^{\tau(t)}z)=e^{-\delta(t)\tau(t)}\Psi_{t}(z).\]
We define
\[\Xi_{t}(z)=-\frac{1}{2}\frac{\delta(t)}{z}P(z,t).\]
Then $\Xi_{t}(z)$ satisfies
\[\Xi_{t}(e^{\tau(t)}z)=\Xi_{t}(z)+\delta(t)\dot{\tau}(t).\]
Let
\[\Upsilon_{t}(z)=\sum_{k=-\infty}^{\infty} \frac{1}{e^{k\tau(t)}z-\xi(t)}-\frac{1}{e^{k\tau(t)}c-\xi(t)}\]
for some $c\in\mathbb{R}$.
Then $\Upsilon_{t}(z)$ satisfies
\[\Upsilon_{t}(e^{\tau(t)}z)=\Upsilon_{t}(z)+  \sum_{k=-\infty}^{\infty} \frac{1}{e^{(k+1)\tau(t)}c-\xi(t)}-\frac{1}{e^{k\tau(t)}c-\xi(t)}.\]
Note that
\begin{eqnarray*}
\sum_{k=-\infty}^{\infty} \frac{1}{e^{(k+1)\tau(t)}c-\xi(t)}-\frac{1}{e^{k\tau(t)}c-\xi(t)}&=& \lim_{N\rightarrow\infty}\sum_{k=-N}^{N} \frac{1}{e^{(k+1)\tau(t)}c-\xi(t)}-\frac{1}{e^{k\tau(t)}c-\xi(t)} \\
&=& \lim_{N\rightarrow\infty}\left[\frac{1}{e^{(N+1)\tau(t)}c-\xi(t)}-\frac{1}{e^{-N\tau(t)}c-\xi(t)}\right]\\
&=&  \frac{1}{\xi(t)}.
\end{eqnarray*}
Thus
\[\Upsilon_{t}(e^{\tau(t)}z)=\Upsilon_{t}(z)+  \frac{1}{\xi(t)}.\]
We let $\delta(t)=\frac{1}{\xi(t)\dot{\tau}(t)}$. Then, as in the proof of Theorem \ref{genslitRS}, $\Theta(t)=\Xi_{t}(z)-\Upsilon_{t}(z)$ is a simple automorphic function and hence does not depend on $z$. Thus $\Theta_{t}(z)=\lambda^{*}(t)$ and $\lambda^{*}(t)$ must be real. Hence
\begin{eqnarray*}&&\Xi_{t}(z)-\Upsilon_{t}(z)=\lambda^{*}(t)
\\\Rightarrow&& P(z,t)=-2z\dot{\tau}(t)\xi(t)\left[\lambda^{*}(t)+\Upsilon_{t}(z)\right].
\end{eqnarray*}
Thus we have proved the following theorem
\begin{theorem}\label{annulus}
For almost all $t\in(0,T)$, $f_{t}(z)$ satisfies
\[\dot{f}_{t}(z)=f_{t}'(z)P(z,t)\]
where
\[P(z,t)=\lambda(t)z-\sum_{k\in\mathbb{Z}} \left(\frac{\dot{\tau}(t)\xi(t)z}{e^{k\tau(t)}z-\xi(t)}+\frac{\dot{\tau}(t)\xi(t)z}{e^{k\tau(t)}c-\xi(t)} \right)\]
for some $\lambda(t)\in\mathbb{R}$ and $c\in\mathbb{R}\setminus\{0\}$.
\end{theorem}
The difference between this case and the general case is that up to reparameterization, $\Gamma_{t}$, does not depend on the curve $\gamma$; this is because, up to conformal equivalence, the set of doubly-connected domains forms a one parameter family of domains. For general Riemann surfaces, the family of Fuchsian groups $\{\Gamma_{t}\}$ depends on the curve $\gamma$.
\section{The solution to the Loewner differential equation}
Firstly, recall that if $f_{t}$ satisfies
\[\dot{f}_{t}(z)=f_{t}'(z)P(z,t),\]
then the inverse function $g_{t}=f_{t}^{-1}$ satisfies
\[\dot{g}_{t}(z)=-P(g_{t}(z),t).\]
In this section, we will prove that given a finitely-generated Fuchsian group $\Gamma$ of the second kind with no parabolic or elliptic elements and $P(z,t)$ as in Theorem \ref{genslitRS}, the differential equation
\begin{equation}\dot{g}_{t}(z)=-P(g_{t}(z),t)\text{
with initial condition } g_{0}(z)=z\label{LDE}\end{equation}
 has solution $g_{t}(z)$ and each function $g_{t}$ is a conformal bijection of some domain $H_{t}\subset\mathbb{H}$ onto $\mathbb{H}$. Moreover $H_{t}$ satisfies
\[\phi(H_{t})=H_{t} \text{ for all } \phi\in\Gamma\]
i.e. $H_{t}$ is $\Gamma$-invariant. Let $\pi:\mathbb{H}\rightarrow \mathbb{H}/\Gamma$ be the quotient map. Then this implies that $R_{t}=\pi(H_{t})$ is a family of sub-Riemann surfaces of $R=\mathbb{H}\setminus\Gamma$ and also, $K_{t}=\pi(\mathbb{H}\setminus H_{t})$ is a continuously growing family of hulls in $R$. We remark that for fixed $z$, (\ref{LDE}) is an ordinary differential equation and so we are guaranteed a solution up to a certain time $T_{z}$.

We will first consider the case when $\Gamma=\left\langle z\mapsto e^{\tau_{0}}z\right\rangle$ for some $\tau_{0}>0$. Then $P(z,t)$ is given by Theorem \ref{annulus}.
\begin{theorem}\label{converseannulus}
Suppose that $\xi,\lambda\in C([0,T])$ and $\tau\in C^{1}([0,T])$ such that $\xi(t),\tau(t),\dot{\tau}(t)>0$ for all $t\in[0,T]$ and suppose that $c\in\mathbb{R}\setminus\{0\}$.
For fixed $z\in\mathbb{H}$, let $g_{t}(z)$ be the solution to the differential equation
\[\dot{g}_{t}(z)=-\lambda(t)g_{t}(z)+\sum_{k\in\mathbb{Z}} \left(\frac{\dot{\tau}(t)\xi(t)g_{t}(z)}{e^{k\tau(t)}g_{t}(z)-\xi(t)}+\frac{\dot{\tau}(t)\xi(t)g_{t}(z)}{e^{k\tau(t)}c-\xi(t)} \right)\]
with initial condition $g_{0}(z)\equiv z$. Let $T_{z}>0$ be the supremum of all $t$ such that the solution $g_{t}(z)$ is well-defined up to time $t$ with $g_{t}(z)\in\mathbb{H}$. Then the function $g_{t}$ is a conformal mapping of $H_{t}=\{z:T_{z}>t\}$ onto $\mathbb{H}$ such that
for all $k\in\mathbb{Z}$,
\[g_{t}(e^{k\tau(0)}g_{t}^{-1}(z))=e^{k\tau(t)}z.\]
Also, $H_{t}$ is $\left\langle z\mapsto e^{\tau_{0}}z\right\rangle$-invariant.
\end{theorem}
\begin{proof}
For $\zeta,\omega\in\mathbb{H}$, let $\Phi_{t}(\zeta,\omega)=P(\zeta,t)-P(\omega,t)$ where
\[P(z,t)=\lambda(t)z-\sum_{k\in\mathbb{Z}} \left(\frac{\dot{\tau}(t)\xi(t)z}{e^{k\tau(t)}z-\xi(t)}+\frac{\dot{\tau}(t)\xi(t)z}{e^{k\tau(t)}c-\xi(t)} \right).\]
Then for fixed $\omega\in\mathbb{H}$, $\Phi_{t}(\zeta,\omega)$ vanishes when $\zeta=\omega$ so we can write
\[\Phi_{t}(\zeta,\omega)=(\zeta-\omega)\Psi_{t}(\zeta,\omega),\]
where $\Psi_{t}(\omega,\omega)\neq 0$.

Let $\Delta_{t}(z,w)=g_{t}(z)-g_{t}(w)$ for $z,w\in H_{t}$. Then
\[\dot{\Delta}_{t}(z,w)=-\Delta_{t}(z,w)\Psi_{t}(g_{t}(z),g_{t}(w)).\]
Since $\Delta_{0}(z,w)=z-w$, this differential equation has solution
\[\Delta_{t}(z,w)=(z-w)\exp\left[-\int_{0}^{t}\Psi_{t}(g_{t}(z),g_{t}(w))dt\right].\]
This implies that $g_{t}(z)$ is continuous with respect to $z$. Then the fact that $P(z,t)$ does not have any poles in $\mathbb{H}$ implies that $g_{t}(z)$ is one-to-one, and analytic. Also, by definition, $g_{t}(H_{t})\subset \mathbb{H}$. We need to show that $g_{t}(H_{t})= \mathbb{H}$.

Fix $t_{0}\in[0,T]$ and take any $w\in \mathbb{H}$, we will find $w'\in H_{t_{0}}$ such that $g_{t_{0}}(w')=w$. To do this, we start with the point $w$ and run the inverse flow of the differential equation starting from $t_{0}$ i.e.
\[\dot{h}_{t}(w)=P(h_{t}(w),t_{0}-t)\]
with initial condition $h_{0}(w)\equiv w$. We need to show that a solution exists for all $0\leq t\leq t_{0}$ and that the solution remains within the upper half-plane. Note that the only way a solution ceases to exist is if, for some $t\in(0,t_{0})$, $h_{t}(w)=e^{k\tau(t_{0}-t)}\xi(t_{0}-t)$ for some $k\in\mathbb{Z}$ or $h_{t}(w)=0$. However, if  $h_{t}(w)$ is close to $\xi(t_{0}-t)$, then
\[\mathrm{Im}[\dot{h}_{t}(w)]\asymp\frac{\dot{\tau}(t)\xi(t)^2\mathrm{Im}[h_{t}(w)]}{|h_{t}(w)-\xi(t_{0}-t)|^{2}}>0.\]
Thus $h_{t}(w)$ cannot approach $\xi(t_{0}-t)$. Similarly for the points $e^{k\tau(t_{0}-t)}\xi(t_{0}-t)$. This also implies that $h_{t}(w)$ cannot approach $0$ since the points $e^{k\tau(t_{0}-t)}\xi(t_{0}-t)$ accumulate at $0$. Hence the solution exists for all $t\in[0,t_{0}]$.  Also, for any $t\in[0,t_{0}]$,
\[\mathrm{Im}[\dot{h}_{t}(\zeta)]=\mathrm{Im}[P(\zeta,t)]=0\]
for all $\zeta\in\mathbb{R}\setminus\{0\}$ with $\zeta\neq e^{k\tau(t)}\xi(t)$ for all $k\in\mathbb{Z}$. Thus by connectivity this implies that $h_{t}(w)\in\mathbb{H}$ for all $t\in[0,t_{0}]$.
Hence a solution $h_{t}(w)$ to the inverse flow starting at $w$ exists up to time $t_{0}$ and remains in $\mathbb{H}$. Let $w'=h_{t_{0}}(w)$. Then note that $h_{t_{0}-t}(w)$ is a solution to
\[\dot{g}_{t}(w')=-P(g_{t}(z),t)\]
with initial condition $g_{0}(w')=w'$. Hence by uniqueness of solution of ordinary differential equations, we must have $g_{t}(w')=h_{t_{0}-t}(w)$ and in particular $g_{t_{0}}(w')=h_{0}(w)=w$.

It only remains to show that
\[g_{t}(e^{k\tau(0)}g_{t}^{-1}(z))=e^{k\tau(t)}z.\]
Firstly, $P(z,t)$ satisfies, for all $k\in\mathbb{Z}$,
\[P(e^{k\tau(t)}z,t)=e^{k\tau(t)}P(z,t)-k\dot{\tau}(t)e^{k\tau(t)}z.\]
This implies that if we let $\widetilde{g}_{t}(z)=e^{k\tau(t)}g_{t}(z)$ (for some $k\in\mathbb{Z}$), then
\[\dot{\widetilde{g}}_{t}(z)=-P(\widetilde{g}_{t}(z),t)\]
with $\widetilde{g}_{0}(z)=e^{k\tau(0)}z$. By the uniqueness of solution of ordinary differential equations, we have
\[\widetilde{g}_{t}(z)=g_{t}(e^{k\tau(0)}z\]
and this implies that
\[g_{t}(e^{k\tau(0)}g_{t}^{-1}(z))=e^{k\tau(t)}z.\]
This also implies that $H_{t}$ is $\left\langle z\mapsto e^{\tau_{0}}z\right\rangle$-invariant.
\end{proof}
When $\Gamma$ is a non-elementary Fuchsian group (i.e. the limit set contains infinitely many points), the situation is more complicated. The reason is that in the previous case, $\Gamma_{t}$ is does not depend on $\xi(t)$ and $\lambda(t)$ (up to reparameterization); this is not true for non-elementary $\Gamma$. This corresponds to the fact that the Teichm\"{u}ller space of an annulus is one-dimensional (see \cite[Chapter VI]{LecturesQC} for more details on Teichm\"{u}ller spaces).

Thus if we want the solution of a differential equation to be a Loewner chain compatible with $\mathbb{H}/\Gamma$ for a finitely-generated non-elementary Fuchsian group of the second kind $\Gamma$ with no parabolic or elliptic elements, we need to determine \emph{beforehand} the family $(\Gamma_{t})_{t\in[a,b]}$ of Fuchsian groups that we want the Loewner chain to correspond to. This is akin to finding the path in Teichm\"{u}ller space that corresponds to the evolution described by the differential equation.

Recall that any M\"{o}bius transformation that preserves $\mathbb{H}$ is specified by 3 real parameters (see e.g. \cite[p. 77]{MR1393195}). For example, if we take 2 ordered triples of distinct points in $\mathbb{R}$, $(p_{1},p_{2},p_{3})$ and $(q_{1},q_{2},q_{3})$, then there is a unique M\"{o}bius transformation $\phi$ satisfying
\[\phi(p_{j})=q_{j}\]
for $j=1,2,3$ and $\phi(\mathbb{H})=\mathbb{H}$. Moreover, the formula for $\phi(z)$ can be determined using cross-ratios (see \cite[Section 4.4]{MR1393195}), namely
\begin{equation}\frac{(q_{1}-q_{3})(q_{2}-\phi(z))}{(q_{1}-q_{2})(q_{3}-\phi(z))}= \frac{(p_{1}-p_{3})(p_{2}-z)}{(p_{1}-p_{2})(p_{3}-z)}\label{heq}.\end{equation}
The idea is to pick 3 distinct points $p_{1},p_{2},p_{3}$ in $\mathbb{R}$ and study their evolution as well as the evolution of their images under elements of $\Gamma$ given by $P(z,t)$. This will allow us to determine the appropriate family of Fuchsian groups $\{\Gamma_{t}\}$ which we can then use with the Loewner differential equation.

Suppose that $\Gamma=\langle\psi_{1},\cdots,\psi_{m}\rangle$ is a finitely-generated Fuchsian group of the second kind. For distinct $p_{1},p_{2},p_{3}\in\mathbb{R}\setminus\Lambda$, let $\mathcal{P}=\{p_{l,j}(t):l=0,\ldots,m \text{ and } j=1,2,3 \}$ be a family of functions such that:
\begin{itemize}
 \item for $l=1,\ldots,m$ and $j=1,2,3$, $p_{l,t}:[0,T]\rightarrow \mathbb{R}$ is a continuous function for some $T>0$;
\item for $j=1,2,3$, $p_{0,j}(0)=p_{j}$ and $p_{l,j}(0)=\psi_{l}(p_{j})$ for $l=1,\ldots, m$.
\end{itemize}
Then $\mathcal{P}$ defines a family of groups $\{\Gamma_{t}\}_{t\in[0,T]}$ such that each $\Gamma_{t}$ is the group generated by $\psi_{1,t},\ldots \psi_{m,t}$ where $\psi_{l,t}$ is the unique M\"{o}bius transformation that satisfies
\[\psi_{l,t}(p_{0,j}(t))=p_{l,j}(t)\text{ and } \psi_{l,t}(\mathbb{H})=\mathbb{H} \text{ for all } l=1\ldots,m \text{ and } j=1,2,3.\]
In particular, $\Gamma_{0}=\Gamma$.
\begin{lemma}\label{lemlast}
Suppose that $\Gamma$ is a finitely-generated Fuchsian group with no parabolic or elliptic elements. Then for any family $\mathcal{P}$ with $\Gamma_{t}$ as above, there exists $S\in(0,T]$ such that for $t\in[0,S]$, $\Gamma_{t}$ is a Fuchsian group of the second kind.
\end{lemma}
\begin{proof}
By considering fixed points or traces (see \cite[Section 4.3]{MR1393195}), it is easy to see that  a sequence of elliptic M\"{o}bius transformations $(\rho_{n})$ that preserve $\mathbb{H}$ cannot converge to a hyperbolic M\"{o}bius transformation $\phi$ that preserves $\mathbb{H}$ as $n\rightarrow\infty$. This implies that there exists $S'\in(0,T]$ such that $\Gamma_{t}$ does not contain any elliptic M\"{o}bius transformations. Theorem 8.1 in \cite{MR1393195} then implies that $\Gamma_{t}$ is a Fuchsian group for $t\in[0,S']$.

Also, from the definition of  $\phi_{l,t}(z)$ in terms of cross ratios in (\ref{heq}), it is clear that $\phi_{l,t}(z)$ varies continuously in $t$. This implies that the fixed points of each $\phi_{t}\in\Gamma_{t}$ also vary continuously in $t$. Since the limit set of $\Gamma_{t}$ is the closure of the hyperbolic fixed points of elements in $\Gamma_{t}$ (see Section 2), this implies that we can find $S\in(0,S']$ such that $\Gamma_{t}$ is a Fuchsian group of the second kind  for all $t\in[0,S]$.
\end{proof}
\begin{theorem}\label{converseRS}
Suppose that $\lambda,\xi\in C([0,\infty))$. Let $\Gamma=\langle\psi_{1},\cdots,\psi_{m}\rangle$ be a finitely-generated Fuchsian group of the second kind with no parabolic or elliptic elements such that $\xi(0)$ is not contained in the limit set $\Lambda$ of $\Gamma$. Also, let $c\in\mathbb{R}\setminus\Lambda$.

Take distinct points $p_{1},p_{2},p_{3}\in \mathbb{R}\setminus\Lambda$  such that $p_{1},p_{2},p_{3}\neq\phi(\xi(0))$ for all $\phi\in\Gamma$. Consider the system of differential equations
\[\dot{p}_{l,j}(t)=-P(p_{l,j}(t),t)\]
for $l=0,\ldots,m$ and $j=1,2,3$, where
\normalsize\[P(z,t)=\frac{\lambda(t)+\sum_{\phi_{t}\in\Gamma_{t}}\left[\frac{2\sigma_{t}}{\phi_{t}(z)-\xi(t)} + \sum_{k=1}^{n} \delta_{k}(t)\left(\frac{\dot{\phi}_{t}(z)-\dot{\phi}_{k,t}(c)}{\phi_{t}(z)-\phi_{k,t}(c)}-\frac{\dot{\phi}_{t}(z)}{\phi_{t}(z)-c} \right)  \right]}{\sum_{\phi_{t}\in\Gamma_{t}}\sum_{k=1}^{n}\delta_{k}(t)\left(\frac{\phi_{t}'(z)}{\phi_{t}(z)-\phi_{k,t}(c)}-\frac{\phi_{t}'(z)}{\phi_{t}(z)-c}\right)}\]\normalsize
with initial conditions given by $p_{0,j}(0)= p_{j}$ and $p_{l,j}(0)=\psi_{l}(p_{j})$ for $l=1,\ldots m$ and $j=1,2,3$. Here
\[\Gamma_{t}=\left\langle\psi_{1,t},\ldots,\psi_{m,t}\right\rangle,\]
and $\psi_{l,t}$ is the unique M\"{o}bius transformation that preserves $\mathbb{H}$ and satisfies $\psi_{l,t}(p_{0,j}(t))=p_{l,j}(t)$ for $l=1,\ldots,m$ and $j=1,2,3$; $\phi_{1,t},\ldots,\phi_{n,t}$ are the hyperbolic free generators of $\Gamma_{t}$; $\delta_{1}(t),\ldots, \delta_{n}(t)$ are chosen such that the sum in the numerator of $P(z,t)$ converges; and
\[\sigma_{t}=-\sum_{\phi_{t}\in\Gamma_{t}}\sum_{k=1}^{n}\delta_{k}(t)\left(\frac{\phi_{t}'(\xi(t))}{\phi_{t}(\xi(t))-\phi_{k,t}(c)}-\frac{\phi_{t}'(\xi(t))}{\phi_{t}(\xi(t))-c}\right).\]
  Note that $P(z,t)$ can be given explicitly in terms of $\xi(t)$, $\lambda(t)$ and  $p_{l,j}(t)$ for $l=0,\ldots,m$ and $j=1,2,3$.

Let $T_{p_{1},p_{2},p_{3}}$ be the largest time such that the above system of differential equations has solution $\{p_{l,j}(t)\}$ for $l=0,\ldots,m$ and $j=1,2,3$ and let
\[T=\sup_{(p_{1},p_{2},p_{3})} T_{p_{1},p_{2},p_{3}},\]
where the supremum is taken over all triples $(p_{1},p_{2},p_{3})$ of distinct points with $p_{1},p_{2},p_{3}\in \mathbb{R}\setminus\Lambda$ and $p_{1},p_{2},p_{3}\neq\phi(\xi(0))$ for all $\phi\in\Gamma$. Then $\Gamma_{t}$ is well-defined for $t\in[0,T]$. Suppose that $g_{t}(z)$ is the solution to the differential equation
\[\dot{g}_{t}(z)=-P(g_{t}(z),t)\]
with initial condition $g_{0}(z)\equiv z$. Let $T_{z}$ be the supremum of all $t\in[0,T]$ such that the solution is well-defined up to time $t$ with $g_{t}(z)\in\mathbb{H}$. Then, for $t\in[0,T]$, the function $g_{t}$ is a conformal map of $H_{t}=\{z:T_{z}>t\}$ onto $\mathbb{H}$ such that
\[g_{t}\circ\Gamma \circ g_{t}^{-1}=\Gamma_{t}\]
for all $t<T$. Also, $H_{t}$ is $\Gamma$-invariant.
\end{theorem}
\begin{proof}
First note that $\delta_{1}(t),\ldots,\delta_{n}(t)$ can be determined by solving the system of linear equations given in Lemma \ref{lem8}, and, in particular, they can be given in terms of $\xi(t)$, $\lambda(t)$ and  $p_{l,j}(t)$ for $l=0,\ldots,m$ and $j=1,2,3$.

Now, Lemma \ref{lemlast} implies that we must have $T_{p_{1},p_{2},p_{3}} >0$. Then for $t\in [0,T_{p_{1},p_{2},p_{3}}]$, the same argument as in the proof of Theorem \ref{converseannulus} implies that $g_{t}(z)$ is continuous, one-to-one, and differentiable and $g_{t}(H_{t})\subset \mathbb{H}$. Note that, $P(z,t)$ satisfies
\begin{equation}P(\phi_{t}(z),t)=\phi_{t}'(z)P(z,t)-\dot{\phi}_{t}(z) \text{ for }\phi_{t}\in\Gamma_{t}.\label{lek}\end{equation}
Since the residue of $P(z,t)$ at $z=\xi(t)$ is equal to $-2$, (\ref{lek}) implies that the residue of $P(z,t)$ at each point in the $\Gamma_{t}$-orbit of $\xi(t)$ is also negative. Hence, by considering the inverse flow as in the proof of Theorem \ref{converseannulus}, we can show that $g_{t}(H_{t})=\mathbb{H}$.

Now, for some $l=1,\ldots,m$, let $\widetilde{g}_{t}(z)=\psi_{l,t}(g_{t}(z))$. Then (\ref{lek}) implies that
\[\dot{\widetilde{g}}_{t}(z)=-P(\widetilde{g}_{t}(z),t)\]
and $\widetilde{g}_{0}(z)=\psi_{l}(z)$. Then the  uniqueness of solution of ordinary differential equations implies that
\[\widetilde{g}_{t}(z)=g_{t}(\psi_{l}(z))\]
and hence
\[g_{t}\circ\psi_{l}\circ g_{t}^{-1}=\psi_{l,t}\]
for $l=1,\ldots,m$. This implies that $\Gamma_{t}=g_{t}\circ \Gamma \circ g_{t}^{-1}$ and also that $H_{t}$ is $\Gamma$-invariant.

Finally, we need to show that if we choose two triples $(p_{1},p_{2},p_{3})$ and $(\widetilde{p}_{1},\widetilde{p}_{2},\widetilde{p}_{3})$, then we generate the same $\Gamma_{t}$ and $g_{t}$ for $0<t<\min(T_{p_{1},p_{2},p_{3}},T_{\widetilde{p}_{1},\widetilde{p}_{2},\widetilde{p}_{3}})$. So suppose that we run the above argument with $(p_{1},p_{2},p_{3})$ and $(\widetilde{p}_{1},\widetilde{p}_{2},\widetilde{p}_{3})$ to obtain $\{p_{l,j}(t)\}$ and $\{\widetilde{p}_{l,j}(t)\}$  respectively and we use $\{p_{l,j}(t)\}$ to obtain $\Gamma_{t}$ and $g_{t}$ as above. Then we note that $g_{t}(\widetilde{p}_{j})$ satisfies the same differential equation as $\widetilde{p}_{0,j}(t)$ and $g_{t}(\psi_{l}(\widetilde{p}_{j}))$ satisfies the same differential equation as $\widetilde{p}_{l,j}(t)$ for each $l=1,\ldots,m$ and $j=1,2,3$. This implies that for $l=1,\ldots,m$ and $j=1,2,3$, we must have
\[ g_{t}(\widetilde{p}_{j})=\widetilde{p}_{0,j}(t),\]
\[g_{t}(\psi_{l}(\widetilde{p}_{j}))=\widetilde{p}_{l,j}(t).\]
Then the fact that
$\Gamma_{t}=g_{t}\circ \Gamma \circ g_{t}^{-1}$ implies that the family of Fuchsian groups generated by $(p_{1},p_{2},p_{3})$ and $(\widetilde{p}_{1},\widetilde{p}_{2},\widetilde{p}_{3})$ are both the same and hence they both give rise to the same conformal map $g_{t}$. Hence we can define $\Gamma_{t}$ and $g_{t}$ for all $t\in[0,T]$, where
\[T=\sup_{p_{1},p_{2},p_{3}} T_{p_{1},p_{2},p_{3}}.\]
where the supremum is taken over all triples $(p_{1},p_{2},p_{3})$ of distinct points with $p_{1},p_{2},p_{3}\in \mathbb{R}\setminus\Lambda$ and $p_{1},p_{2},p_{3}\neq\phi(\xi(0))$ for all $\phi\in\Gamma$.
\end{proof}
\begin{remark}\
\begin{enumerate}
\item[(i)] As with the simply-connected versions of the Loewner differential equation, we note that the Riemann surfaces $\pi(H_{t})$ (where $\pi:\mathbb{H}\rightarrow \mathbb{H}\setminus\Gamma$ is the quotient map) are not necessarily obtained from $\mathbb{H}/\Gamma$ by removing a curve growing from the boundary (see the example of a spiral constructed by Marshall and Rohde in \cite{MR2163382}).
\item[(ii)] If $\Gamma$ is a Fuchsian group of the second kind with no elliptic elements such that $\mathbb{H}/\Gamma$ is a Riemann surface with boundary, we can obtain the same result by approximating $\Gamma$ with Fuchsian groups of the second kind that only contain hyperbolic M\"{o}bius transformations.
\end{enumerate}
\end{remark}
\section{Some remarks on stochastic Loewner evolution}
As mentioned in the introduction, the motivation of the results in this paper is that we want to define a version of stochastic Loewner evolution on multiply-connected domains and Riemann surfaces. We now give an informal overview of how this can be done. Firstly, note that in the simply-connected case, there are two natural normalizations and parameterizations for the conformal mappings $f_{t}$ (or $g_{t}=f_{t}^{-1}$) which result in the chordal and radial versions of the Loewner differential equation (see \cite[Sections 4.1 and 4.2]{MR2129588}). Let us restrict our attention to the chordal Loewner differential equation:
\[\dot{g}_{t}(z)=\frac{2}{g_{t}(z)-\xi(t)}.\]
Suppose that $\gamma:[0,\infty]\rightarrow\overline{\mathbb{H}}$ is a simple curve with $\gamma(0,\infty)\in\mathbb{H}$ and $\gamma(0)=0,\gamma(\infty)=\infty$. If $\gamma$ is conformally invariant, then the driving function of the chordal Loewner differential equation $\xi(t)$ is forced to be $\sqrt{\kappa}B_{t}$ for some $\kappa\geq0$ where $B_{t}$ is standard 1-dimensional Brownian motion (see \cite[p. 147]{MR2129588}). We define the solution $g_{t}$ to the chordal Loewner differential equation with $\xi(t)=\sqrt{\kappa}B_{t}$ to be chordal stochastic Loewner evolution with parameter $\kappa$ (chordal SLE$_{\kappa}$). It can be shown that, for any $\kappa\geq0$, chordal SLE$_{\kappa}$ is almost surely generated by a curve $\gamma:[0,\infty)\rightarrow \overline{\mathbb{H}}$  that does not cross itself (see \cite[Theorem 6.3]{MR2129588}).

Now suppose that $\Gamma$ is a Fuchsian group preserving $\mathbb{H}$ such that $\mathbb{H}/\Gamma$ is a Riemann surface with boundary. In this case, there is typically no ``natural'' normalization for the Loewner chain corresponding to a curve $\gamma$ on $\mathbb{H}/\Gamma$. In addition, the argument in the simply-connected case no longer works and hence there is no ``fixed'' choice for the functions $\xi(t)$ and $\lambda(t)$ in Theorem \ref{converseRS}. However, we do expect a SLE on a Riemann surface to ``look like'' a chordal SLE$_{\kappa}$ and this suggests that $\xi_{t}=\xi(t)$ should satisfy the stochastic differential equation
\[d\xi_{t}=\sqrt{\kappa}dB_{t} + h(t)dt\]
for some function $h(t)$. Thus we can define many versions of stochastic Loewner evolutions on $\mathbb{H}/\Gamma$ by choosing different functions for $h(t)$ and $\lambda(t)$. SLE can then be defined on any Riemann surface $R$ that is conformally equivalent to $\mathbb{H}/\Gamma$. The problem then becomes one of finding the right $h(t)$ and $\lambda(t)$ for a particular application.

\bibliographystyle{plain}

\end{document}